\documentclass[12pt]{amsart}
\usepackage{amsfonts,mathrsfs,amsmath,amssymb,amscd,mathtools}
\usepackage{graphicx}
\usepackage{caption}
\usepackage{color,xcolor}
\usepackage[percent]{overpic}
\usepackage[all]{xy}
\usepackage{hyperref}
\usepackage{leftidx}

\newtheorem{thm}{Theorem}[section]
\newtheorem{prop}[thm]{Proposition}
\newtheorem{lemma}[thm]{Lemma}

\newtheorem{cor}[thm]{Corollary}

\theoremstyle{definition}

\newtheorem{definition}[thm]{Definition}
\newtheorem{remark}[thm]{Remark}

\title{Holomorphic Lagrangian branes correspond to perverse sheaves}

\author[Xin Jin]{Xin Jin}
\email{xjin@math.berkeley.edu}

\address{Department of Mathematics, University of California, Berkeley}

\subjclass[2000]{}

\date{}


\numberwithin{equation}{section}

\begin{document}
\begin{abstract}
Let $X$ be a compact complex manifold, $D_c^b(X)$ be the bounded derived category of constructible sheaves on $X$, and $Fuk(T^*X)$ be the Fukaya category of $T^*X$. A Lagrangian brane in $Fuk(T^*X)$ is holomorphic if the underlying Lagrangian submanifold is complex analytic in $T^*X_{\mathbb{C}}$, the holomorphic cotangent bundle of $X$. We prove that under the quasi-equivalence
between $D^b_c(X)$ and $DFuk(T^*X)$ established in \cite{NZ} and \cite{Nadler2}, holomorphic Lagrangian branes with appropriate grading correspond to perverse sheaves.
\end{abstract}

\maketitle
\tableofcontents

\section{Introduction}
For a real analytic manifold $X$, one could consider two invariants that encode the local/global analytic and topological structure of $X$, one is the derived category $D^b_c(X)$ of constructible sheaves on $X$, and the other is the Fukaya category $Fuk(T^*X)$ of its cotangent bundle $T^*X$. Roughly speaking,  $D^b_c(X)$ is generated by locally constant sheaves supported on submanifolds of $X$, which we will call (co)standard sheaves. The morphism spaces between these sheaves are naturally identified with relative singular cohomology of certain subsets of $X$ taking values in local systems. On the other hand, $Fuk(T^*X)$ is a realm of studying exact Lagrangian submanifolds of $T^*X$ and the intersection theory of them. Here we use the infinitesimal Fukaya category from \cite{NZ}, where Lagrangian branes are allowed to be noncompact and should have controlled behavior near infinity.

In \cite{NZ}, Nadler and Zaslow established a canonical quasi-embedding 
$$H^0(\mu_X): D_c^b(X)\hookrightarrow DFuk(T^*X)$$ induced from $\mu_X$, which is called the microlocal functor, between the $A_\infty$-version of these two categories. Later on, Nadler \cite{Nadler2} proved that $\mu_X$ is actually a quasi-equivalence of categories, hence $H^0(\mu_X)$ is an equivalence. The key ingredient in the construction of $\mu_X$ is to associate each standard or costandard sheaf  a Lagrangian brane in $T^*X$ which lives over the submanifold and asymptotically approaches the singular support of the sheaf near infinity, so that the Floer cohomologies for these branes match with the morphisms on the sheaf side. One could view this as a way of quantizing the singular support of a sheaf by a Lagrangian brane. 

In the complex setting, when $X$ is a complex manifold, one could also study $\mathcal{D}$-modules on $X$. The Riemann-Hilbert correspondence equates the derived category of regular holonomic $\mathcal{D}$-modules $D_{rh}^b(\mathcal{D}_X\text{-mod})$ with $D_c^b(X)$. There are also physical interpretions of the relation of branes (including coisotropic branes) with $\mathcal{D}$-modules, see \cite{Kapustin}, \cite{KW}. These relations together with $\mu_X$ connect different approaches to quantizing conical Lagrangians in $T^*X$.  

In this paper, we investigate the special role of holomorphic Lagrangian branes in $Fuk(T^*X)$ in the complex setting, via the Nadler-Zaslow correspondence. For the notion of holomorphic, we have used the complex structure on $T^*X$ induced from that on $X$. Recall there is an abelian category sitting inside $D_c^b(X)$, the category of perverse sheaves, which is the image of the standard abelian category (single $\mathcal{D}$-modules) in $D_{rh}^b(\mathcal{D}_X\text{-mod})$ under the Riemann-Hilbert correspondence. Our main result is the following:

\begin{thm}\label{main}
Let $X$ be a compact complex manifold and $H^0(\mu_X)^{-1}$ denote the inverse functor of $H^0(\mu_X)$. Then for any holomorphic Lagrangian brane $L$ in $T^*X$, $H^0(\mu_X)^{-1}(L)$ is a perverse sheaf in $D^b_c(X)$ up to a shift. Equivalently, $L$ gives rise to a single holonomic $\mathcal{D}$-module on $X$. 
\end{thm}

In the remainder of the introduction, we discuss the motivation and the proof of our result from two aspects: the symplectic geometry and the microlocal geometry. In the symplectic geometry part, we will summarize the Floer cohomology calculations we have for certain classes of Lagrangian branes. Then in the microlocal geometric side, we will introduce the microlocal approach to perverse sheaves and explain why the Floer calculations imply our main theorem. All of the functors below are derived and we will always omit the derived notation $R$ or $L$ unless otherwise specified.
 
\subsection{Floer complex calculations}\label{HF}

We calculate the Floer complex for two pairs of Lagrangian branes in the cotangent bundle $T^*X$ of a complex manifold $X$. It involves three kinds of Lagrangians which we briefly describe. Firstly, we have a (exact) holomorphic Lagrangian brane $L$ with grading $-\dim_\mathbb{C}X$ (see Lemma \ref{grading}). One could dilate $L$ using the $\mathbb{R}_+$-action on the cotangent fibers and take limit to get a conical Lagrangian
\begin{equation}\label{ConicIntro}
\text{Conic}(L):=\lim\limits_{t\rightarrow 0+}t\cdot L.
\end{equation}
Then for each smooth point $(x,\xi)\in \text{Conic}(L)$, we define a Lagrangian brane, which we will call a \emph{local Morse brane}, depending on the following data. We choose a generic holomorphic function $F$ near $x$ which vanishes at $x$ and has $d\Re(F)_x=\xi$. By the word``generic", we mean the graph $\Gamma_{d\Re(F)}$ should intersect $\text{Conic}(L)$ at $(x,\xi)$ in a transverse way. Then the local Morse brane, denoted as $L_{x,F}$, is defined by extending $\Gamma_{d\Re(F)}$ in an appropriate way, so that $L_{x,F}$ lives over a small neighborhood of $x$, and $L_{x,F}$ has certain behavior near infinity. Note that the construction of $L_{x,F}$ is completely local; it only knows the local geometry (actually the microlocal geometry) around $x$. The last kind of Lagrangian we consider is the brane corresponding to a standard sheaf associated to an open set $V$ under the microlocal functor $\mu_X$. The construction is very easy. Take a function $m$ on $X$ with $m=0$ on $\partial V$ and $m>0$ on $V$, then the Lagrangian is the graph $\Gamma_{d\log m}$, which lives over $V$. We will call such a brane a \emph{standard brane} and denote it by $L_{V,m}$. We have been mixing up the terminology Lagrangian and Lagrangian brane freely, since the Lagrangians $L_{x,F}$ and $L_{V,m}$ will be equipped with canonical brane structures. 
Our Floer complex calculations show the following: 
\begin{thm}
Under certain assumptions on the boundary of $V$, we have
\begin{equation}\label{HF1} 
HF(L_{x,F}, L_{V,m})\simeq (\Omega (B_\epsilon(x)\cap V, B_\epsilon(x)\cap V\cap\{\Re(F)<0\}),d),
\end{equation}
\begin{equation}\label{HF2}
HF^\bullet(L_{x,F},L)=0\text{ for }\bullet\neq 0,
\end{equation}
where $B_\epsilon(x)$ is a small ball around $x$ and the first identification is a canonical quasi-isomorphism. 
\end{thm}
An illustrating picture \footnote{Of course $X=\mathbb{R}$ is not in a complex manifold, but it will become clear that the construction of local Morse brane generalizes to the real setting; also see Section \ref{LMGF}} for the branes $L_{U,m}$ and $L_{x,F}$ in the case of $X=\mathbb{R}$ is presented in Figure \ref{morse}, where $V=(a,b)$, $F_1=x-b$ and $F_2=b-x$. The standard brane $L_{V,m}$ corresponds to the sheaf $i_*\mathbb{C}_V$, and one can check that (\ref{HF1}) holds and compare it with (\ref{Mstandard}).

\begin{figure}
\centering
  \begin{overpic}[width=7cm, height=6cm]{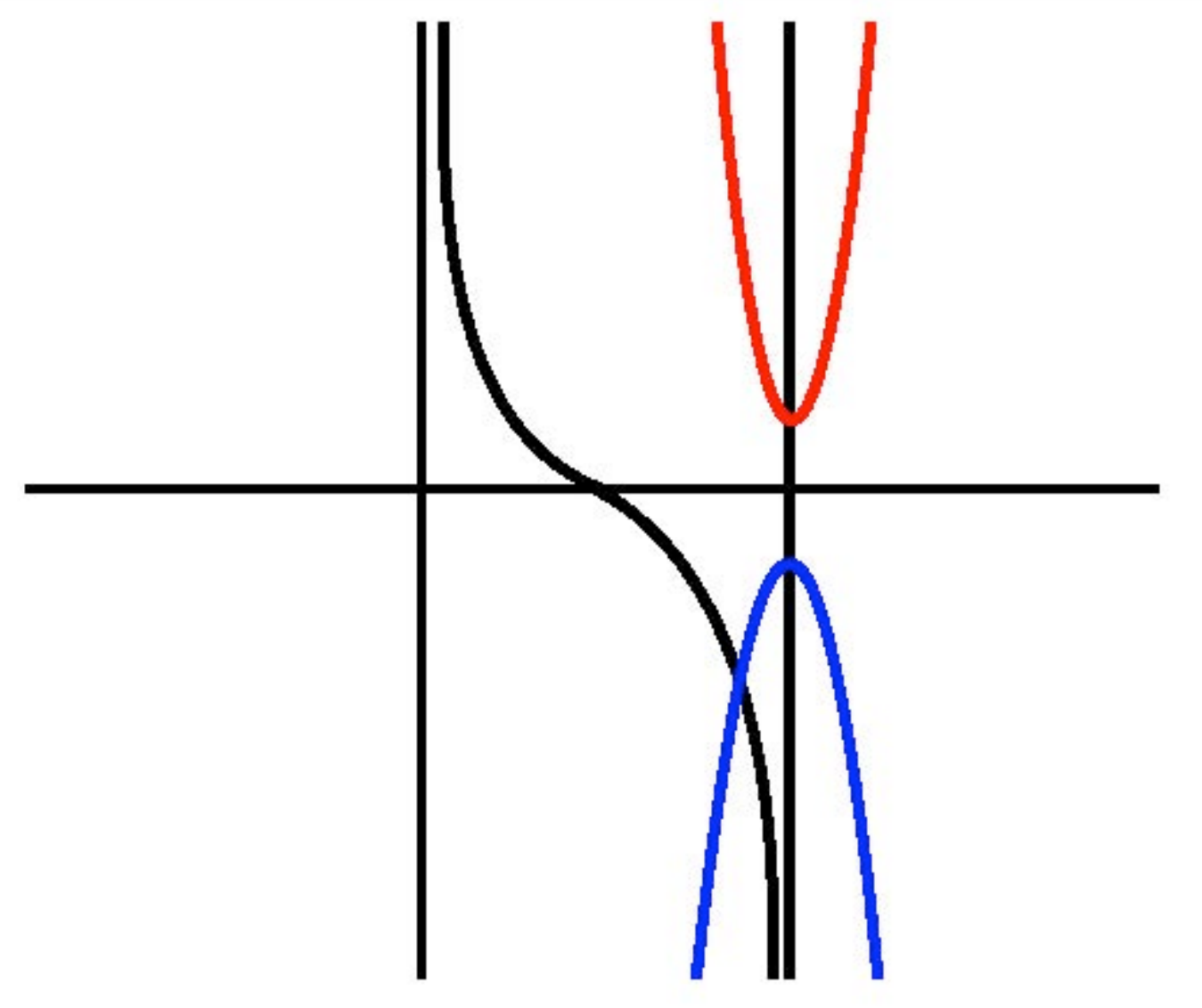}
  \put(74,74){\large \color{red}$L_{b,F_1}$}
  \put(47,46){\large $L_{V,m}$}
  \put(74,16){\large \color{blue}$L_{b,F_2}$}
  \put(30,36){\large $a$}
  \put(61,35){\large $b$}
  \put(87,35){\large $\mathbb{R}$}
  \end{overpic}
\caption{A picture illustrating a standard brane, two local Morse branes and their Floer cohomology for $X=\mathbb{R}$.}
\label{morse}
\end{figure}
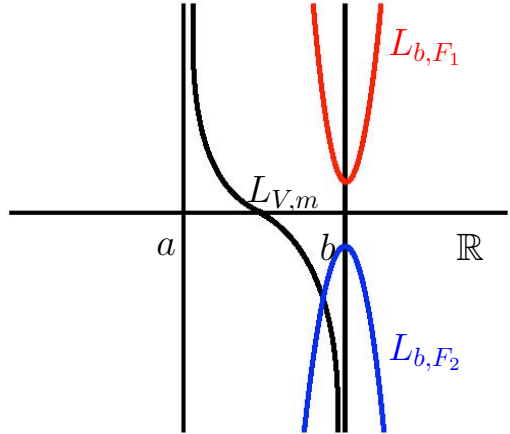

\subsection{Microlocal Geometry}
There are roughly two characterizations of perverse sheaves. One is characterized by the vanishing degrees of the cohomological (co)stalks of sheaves. The other is the microlocal (or Morse theoretic) approach using vanishing property of the microlocal stalks (or local Morse groups) of a sheaf. These are due to Beilinson-Bernstein-Deligne\cite{BBD}, Kashiwara-Schapira\cite{Kashiwara} and Goresky-MacPherson\cite{Goresky}. In this paper, we will mainly adopt the latter one. We also include a path from (co)stalk characterization to the microlocal characterization in Section \ref{LMG}. 

The microlocal stalk of a sheaf is a measurement of the change of sections of the sheaf when propagating along the direction determined by a given covector in $T^*X$. More precisely, let $\mathcal{F}$ be a sheaf whose cohomology sheaf is constructible with respect to some  stratification $\mathcal{S}$. There is the standard conical Lagrangian $\Lambda_\mathcal{S}$ in $T^*X$ associated to $\mathcal{S}$, which is the union of all the conormals to the strata. Now pick a smooth point $(x,\xi)$ in $\Lambda_{\mathcal{S}}$, and choose a sufficiently generic holomorphic function $F$ near $x$ with $F(x)=0$ and $dF_x=\xi$ (this is exactly the same condition we put on $F$ when we construct $L_{x,F}$ in \ref{HF}). The microlocal stalk (or local Morse group) $M_{x,F}(\mathcal{F})$ of $\mathcal{F}$ is defined to be
\begin{align}
M_{x,F}(\mathcal{F})=\Gamma(B_\epsilon(x), B_\epsilon(x)\cap \{\Re(F)<0\},\mathcal{F})
\end{align}
for sufficiently small ball $B_\epsilon(x)$. In particular if $\mathcal{F}$ is $i_*\mathbb{C}_V$, the standard sheaf associated an open embedding $i: V\hookrightarrow X$, then one gets
\begin{equation}\label{Mstandard}
M_{x,F}(i_*\mathbb{C}_V)=\Gamma(B_\epsilon(x)\cap V, B_\epsilon(x)\cap V\cap \{\Re(F)<0\}, \mathbb{C}).
\end{equation}

Recall that $H^0(\mu_X)$ sends $i_*\mathbb{C}_V$ to $L_{V,m}$, and standard sheaves associated to open sets generate the category $D_c^b(X)$. So comparing (\ref{Mstandard}) with (\ref{HF1}), one almost sees that the functor $HF(L_{x,F},-)$ on $DFuk(T^*X)$ is equivalent to the functor $M_{x,F}(-)$ on $D_c^b(X)$ under the Nadler-Zaslow correspondence. This is confirmed by studying composition maps on the $A_\infty$-level.

With the same assumptions as above plus the further assumption that $\mathcal{S}$ is a complex stratification, the microlocal characterization of a perverse sheaf is very simple. It says that $\mathcal{F}$ is a perverse sheaf if and only if the cohomology of the microlocal stalk $M_{x,F}(\mathcal{F})$ is concentrated in degree 0 for all choices of $(x,\xi)$. For a holomorphic Lagrangian brane $L$, it is not hard to prove that $H^0(\mu_X)^{-1}(L)$ is a sheaf whose cohomology sheaf is constructible with respect to a complex stratification. Now it is easy to see that (\ref{HF2}) directly implies our main theorem (Theorem \ref{main}).

\subsection{Organization} The preliminaries are included in the Appendices. We first collect basic material on analytic-geometric categories, since this is a reasonable setting for stratification theory (hence for constructible sheaves) and for Lagrangian branes. Then we give a short account of $A_\infty$-categories, which is the algebra basics for Fukaya category. Lastly, we give an overview of the defininition of infinitesimal Fukaya categories, with some specific account for $Fuk(T^*X)$ to supplement the main content. 

Section \ref{LMG} starts from basic definitions and properties of constructible sheaves and perverse sheaves, then heads towards the microlocal characterization of a perverse sheaf. Section \ref{NZCorr} gives an overview of Nadler-Zaslow correspondence, with detailed discussion on several aspects, including Morse trees and the use of Homological Perturbation Lemma, since similar techniques will be applied in the later sections. Section \ref{Repr M_{x,F}} devotes to the construction of the local Morse brane $L_{x,F}$ and the proof that it corresponds to the local Morse group functor $M_{x,F}$ on the sheaf side. In section \ref{SectionHF}, we  show the proof of (\ref{HF2}) and conclude with our main theorem, some consequences and generalizations. 

\subsection*{Acknowledgements} I would like to express my deep gratitude to my advisor, Prof. David Nadler, for suggesting this problem to me and for numerous invaluable discussions throughout this project. I am indebted to Prof. Denis Auroux, Cheol-Hyun Cho, Si Li and Eric Zaslow for several helpful discussions. I would also like to thank Penghui Li, Zack Sylvan and Hiro Tanaka for useful conversations. \\

\section{Perverse sheaves and the local Morse group functor}\label{LMG}
\subsection{Constructible sheaves}\label{sheaves}
Let $X$ be an analytic manifold. Throughout the paper, we will always work in a fixed analytic-geometric setting, and all the stratifications we consider are assumed to be Whitney stratifications  (see Appendix \ref{agc}). A sheaf of $\mathbb{C}$-vector spaces is \emph{constructible} if there exists a stratification $\mathcal{S}=\{S_\alpha\}_{\alpha\in\Lambda}$ such that $i_\alpha^*\mathcal{F}$ is a locally constant sheaf, where $i_\alpha$ is the inclusion $S_\alpha\hookrightarrow X$. Let $D_c^b(X)$ denote the bounded derived category of complexes of sheaves whose cohomomology sheaves are all constructible. In the following, we simply call such a complex a sheaf. $D_c^b(X)$ has a natural differential graded enrichment, denoted as $Sh(X)$. The morphism space between two sheaves $\mathcal{F}, \mathcal{G}$ is the complex
$RHom(\mathcal{F},\mathcal{G}),$
where $RHom(\mathcal{F}, \cdot )$ is the right derived functor of the usual $Hom(\mathcal{F}, \cdot)$ functor (by taking global section of the sheaf $\mathcal{H}om(\mathcal{F}, \cdot)$). Similarly, we denote by $Sh_\mathcal{S}(X)$ the subcategory of $Sh(X)$ consisting of sheaves constructible with respect a fixed stratification $\mathcal{S}$.

There are the standard four functors between $Sh(X)$ and $Sh(Y)$ associated to a map $f: X\rightarrow Y$, namely $f_*, f_!, f^*$ and $f^!$. Here and after, all functors are derived, though we omit the derived notation. $f_*, f^*$ are right and left adjoint functors, similar for the pair $f^!$ and $f_!$. More explicitly, we have for $\mathcal{F}\in Sh(X), \mathcal{G}\in Sh(Y)$, 
$$\mathcal{H}om(\mathcal{G}, f_*\mathcal{F})\simeq f_*\mathcal{H}om(f^*\mathcal{G}, \mathcal{F}), f_*\mathcal{H}om(\mathcal{F}, f^!\mathcal{G})\simeq \mathcal{H}om(f_!\mathcal{F}, \mathcal{G})$$
$$Hom(\mathcal{G}, f_*\mathcal{F})\simeq Hom(f^*\mathcal{G}, \mathcal{F}), 
Hom(\mathcal{F}, f^!\mathcal{G})\simeq Hom(f_!\mathcal{F}, \mathcal{G}).$$

And there is the Verdier duality $\mathbb{D}: Sh(X)\rightarrow Sh(X)^{op}$, which gives the relation $\mathbb{D}f_*=f_!\mathbb{D}, \mathbb{D}f^*=f^!\mathbb{D}$. Let $i: U\hookrightarrow X$ be an open inclusion and $j: Y=X-U\hookrightarrow X$ be the closed inclusion of the complement of $U$, then $i^*=i^!$ and $j_*=j_!$. There are two standard exact triangles, taking global sections of which gives the long exact sequences for the relative hypercohomology of $\mathcal{F}$ for the pair $(X, Y)$ and $(X,U)$ respectively,
\begin{equation}\label{rel_tr}
 j_!j^!\mathcal{F}\rightarrow\mathcal{F}\rightarrow i_*i^*\mathcal{F}\overset{[1]}{\rightarrow},\  i_!i^!\mathcal{F}\rightarrow\mathcal{F}\rightarrow j_*j^*\mathcal{F}\overset{[1]}{\rightarrow}.
\end{equation}

The \emph{stalk} of $\mathcal{F}$ at $x\in X$ will mean the complex $i_x^*\mathcal{F}$, where $i_x: \{x\}\hookrightarrow X$ is the inclusion. The $i$-th cohomology sheaf of a complex $\mathcal{F}$ will be denoted as $\mathcal{H}^i(\mathcal{F})$. Note that the stalk of $\mathcal{H}^i(\mathcal{F})$ at $x$ is isomorphic to $H^i(i_x^*\mathcal{F})$. Also let $\mathrm{supp}(\mathcal{F}):=\overline{\{x\in X:\mathcal{H}^j(\mathcal{F})_x\neq 0 \text{ for some }j\}}$. 

According to \cite{NZ}, the \emph{standard} objects, i.e. sheaves of the form $i_*\mathbb{C}_U$, where $i: U\hookrightarrow X$ is an open inclusion, generate $Sh(X)$. The argument goes like the following. It suffices to prove the statement for the subcategory $Sh_\mathcal{S}(X)$ for any stratification $\mathcal{S}=\{S_\alpha\}_{\alpha\in\Lambda}$. Without loss of generality, we can assume each stratum of $\mathcal{S}$ is connected and is a cell. Let $\mathcal{S}_{\leq k}$, $0\leq k\leq n=\dim X$, denote the union of all strata in $\mathcal{S}$ of dimension less than or equal to $k$. Let $\mathcal{S}_{>k}=X-\mathcal{S}_{\leq k}$ and $\mathcal{S}_{k}=\mathcal{S}_{\leq k}-\mathcal{S}_{\leq k-1}$. Denote by $i_{k}, i_{>k}, j_{\leq k}$ the inclusion of $S_{\bullet}$ with corresponding subscripts.
The standard exact triangle on the left of (\ref{rel_tr}) for a sheaf $\mathcal{G}$ supported on $\mathcal{S}_{\leq k}$ gives, 
\begin{equation}\label{ex_tr}
 j_{\leq k-1!}j^!_{\leq k-1}\mathcal{G}\rightarrow\mathcal{G}\rightarrow i_{k*}i_{k}^*\mathcal{\mathcal{G}}=i_{>k-1*}i_{>k-1}^*\mathcal{\mathcal{G}}\overset{[1]}{\rightarrow}.
\end{equation}
We start from $\mathcal{G}_n=\mathcal{F}\in Sh_{\mathcal{S}}(X)$, then use (\ref{ex_tr}) inductively for $\mathcal{G}_{k-1}= j_{\leq k-1!}j^!_{\leq k-1}\mathcal{\mathcal{G}}_k= j_{\leq k-1!}j^!_{\leq k-1}\mathcal{F}$ from $k=n$ through $k=1$, and get that $\mathcal{F}$ can be obtained by taking interated mapping cones of shifts of $i_{\alpha*}\mathbb{C}_{S_\alpha}$, $S_\alpha\in\mathcal{S}$. Let $\mathcal{U}_\mathcal{S}=\{X, O_\alpha=X-\overline{S_\alpha}, O'_\alpha=X-\partial S_{\alpha}: \alpha\in\Lambda\}$. Now the claim is $i_{\alpha*}\mathbb{C}_{S_\alpha}$ can be generated by $i_{U*}\mathbb{C}_U, U\in\mathcal{U}_{\mathcal{S}}$. This follows from similar argument. Putting $\mathcal{F}=\mathbb{C}_X$, $i=i_{O_\alpha}$ or $i=i_{O'_\alpha}$ on the left of (\ref{rel_tr}), we get the generation statement for $ j_{\overline{S}_\alpha!}j^!_{\overline{S}_\alpha}\mathbb{C}_X,  j_{\partial S_\alpha!}j^!_{\partial S_\alpha}\mathbb{C}_X$. Then letting $\mathcal{G}=j_{\overline{S}_\alpha!}j^!_{\overline{S}_\alpha}\mathbb{C}_X$ in (\ref{ex_tr}) for $k=\dim S_\alpha$, and identifitying $j_{\leq k-1!}j^!_{\leq k-1}\mathcal{G}$ with $ j_{\partial S_\alpha!}j^!_{\partial S_\alpha}\mathcal{G}$, $i_{k*}i_{k}^*\mathcal{\mathcal{G}}$ with $i_{\alpha*}\mathbb{C}_{S_\alpha}$, we get the generation statement for $i_{\alpha*}\mathbb{C}_{S_\alpha}$.

Since $Sh(X)$ is a dg category, it suffices to study the morphisms between any two standard objects associated to open sets, and the composition maps for a triple of standard objects. 
\begin{prop}[Lemma 4.4.1 \cite{NZ}]\label{open_hom}
Let $i_0: U_0\hookrightarrow X$ and $i_1: U_1\hookrightarrow X$ be the inclusion of two open submanifolds of $X$. Then there is a natural quasi-isomorphism
$$Hom(i_{0*}\mathbb{C}_{U_0}, i_{1*}\mathbb{C}_{U_1})\simeq (\Omega(\overline{U}_0\cap U_1, \partial U_0\cap U_1), d).$$
Furthermore, for a triple of open inclusions $i_k: U_k\hookrightarrow X$, $k=0,1,2$, the composition map 
$$Hom(i_{1*}\mathbb{C}_{U_1}, i_{2*}\mathbb{C}_{U_2})\otimes Hom(i_{0*}\mathbb{C}_{U_0}, i_{1*}\mathbb{C}_{U_1})\rightarrow Hom(i_{0*}\mathbb{C}_{U_0}, i_{2*}\mathbb{C}_{U_2})$$
is natually identified with the wedge product on (relative) deRham complexes
$$(\Omega(\overline{U}_1\cap U_2, \partial U_1\cap U_2), d)\otimes (\Omega(\overline{U}_0\cap U_1, \partial U_0\cap U_1), d)\rightarrow (\Omega(\overline{U}_0\cap U_2, \partial U_0\cap U_2), d)$$
\end{prop}

\cite{NZ} introduced how to perturb $U_0$ and $U_1$ to have transverse boundary intersection, and to use the perturbed open sets to calculate $Hom(i_{0*}\mathbb{C}_{U_0}, i_{1*}\mathbb{C}_{U_1})$. Let $m_i$ be a semi-defining function of $U_i$ for $i=0,1$ (see Remark \ref{def fcn}). There exists a finged set $R\subset \mathbb{R}_+^2$ (see Definition \ref{fringed}) such that $m_1\times m_0: X\rightarrow\mathbb{R}^2$ has no critical value in $R$ (by Corollary \ref{fringed_cor}). In particular, for $(t_1, t_0)\in R$, $X_{m_0=t_0}$ and $X_{m_1=t_1}$ intersect transversely. Then there is a compatible collection of identifications
$$(\Omega(\overline{U}_0\cap U_1, \partial U_0\cap U_1), d)\simeq 
(\Omega(X_{m_0\geq t_0}\cap X_{m_1>t_1}, X_{m_0=t_0}\cap X_{m_1>t_1}), d).$$

\subsection{Perverse sheaves}
Let $X$ be a complex analytic manifold of dimension $n$. In this section, we review some basic definitions and properties of
the perverse $t$-structure $(^p{\!
D}_c^{\leq 0}(X), ^p{\! D}_c^{\geq 0}(X))$ (with respect to the
``middle perversity"). The exposition is following Section 8.1 of \cite{HTT}.

\begin{definition}
Define the full subcategories $^p{\! D}_c^{\leq 0}(X)$ and $^p{\!
D}_c^{\geq 0}(X)$ in $D_c^b(X)$ as follows. A sheaf $\mathcal{F}\in {^p}{\!D}_c^{\leq 0}(X)$ if
$$\dim \{\mathrm{Supp}(\mathcal{H}^j(\mathcal{F}))\}\leq -j,\text{ for all }j\in\mathbb{Z} $$
and $\mathcal{F}\in {^p}{\! D}_c^{\geq 0}(X)$ if
$$\dim
\{\text{Supp}(\mathcal{H}^j(\mathbb{D}\mathcal{F}))\}\leq -j, \text{ for all
}j\in\mathbb{Z}.$$
\end{definition}

An object of its heart $Perv(X)={^p}{\! D}^{\leq 0}(X)\cap {^p}{\!
D}^{\geq 0}(X)$ is called a \emph{perverse sheaf}. Let ${^p}{\!
\tau}^{\leq k}: D_c^b(X)\rightarrow {^p}{\! D}^{\leq k}(X):={^p}{\! D}^{\leq 0}(X)[-k]$, ${^p}{\!
\tau}^{\geq k}: D_c^b(X)\rightarrow {^p}{\! D}^{\geq k}(X):= {^p}{\! D}^{\geq 0}(X)[-k]$ be the
corresponding truncation functors.
Let ${^p}{\! H}^k={^p}{\!\tau}^{\geq k}{^p}{\!\tau}^{\leq k}[k]:
D_c^b(X)\rightarrow Perv(X)$ be the $k$-th perverse cohomology functor.

Here are several properties of perverse $t$-structures.
\begin{prop}\label{perverse}
Let $\mathcal{F}\in D_c^b(X)$ and $\mathcal{S}=\{S_\alpha\}_{\alpha\in\Lambda}$ be
a complex stratification of $X$ consisting of connected strata with respect to which $H^j(\mathcal{F})$ are constructible. Then \\
(i) $\mathcal{F}\in D_c^{\leq 0}(X)$ if and only if $H^j(i_{S_\alpha}^*\mathcal{F})=0$
for all $j>-\dim S_\alpha$;\\
(ii) $\mathcal{F}\in D_c^{\geq} 0(X)$ if and only if $H^j(i_{S_\alpha}^!\mathcal{F})=0$
for all $j<-\dim S_\alpha$.
\end{prop}
\begin{lemma}\label{perverse cohomology}
(1) A sheaf $\mathcal{F}\in D_c^b(X)$ is isomorphic to zero if and
only if ${^p}{\! H}^k(\mathcal{F})=0$
for all $k\in\mathbb{Z}$.\\
(2) A morphism $f: \mathcal{F}\rightarrow \mathcal{G}$ in $D_c^b(X)$
is an isomorphism if and only if the induced map ${^p}{\! H}^k(f):\
^p{\! H}^k(\mathcal{F})\rightarrow{^p}{\! H}{^k}(\mathcal{G})$ is an
isomorphism for all $k\in\mathbb{Z}$.
\end{lemma}
\begin{proof}
(1) Since $\mathcal{F}\in D^b_c(X)$, we have $\mathcal{F}\in {^p}{\!
D}^{\geq a}(X)\cap{^p}{\! D}^{\leq b}(X)$ for some $a,b\in\mathbb{Z}$. By
the distinguished triangle,
$${^p}{\! \tau}^{\leq b-1}\mathcal{F}\rightarrow {^p}{\! \tau}^{\leq b}\mathcal{F}\simeq \mathcal{F}\rightarrow
{^p}{\! \tau}^{\leq b}{^p}{\! \tau}^{\geq b}\mathcal{F}\simeq
0\overset{[1]}{\rightarrow},$$ we get $\mathcal{F}\in {^p}{\!
D}^{\leq b-1}(X)$. Inductively, we conclude that $\mathcal{F}\simeq 0$.

(2) is an easy consequence of (1).
\end{proof}
\begin{prop}\label{perverse t-structure}
$\mathcal{F}\in D_c^b(X)$ is perverse $\Leftrightarrow$ ${^p}{\!
H}{^*}(\mathcal{F})$ is concentrated in degree $0$.
\end{prop}
\begin{proof}
$``\Rightarrow" $ is clear.\\
$``\Leftarrow"$: Consider the exact triangle
$${^p}{\!\tau}{^{\leq -1}}\mathcal{F}\rightarrow\mathcal{F}
\rightarrow {^p}{\!\tau}{^{\geq
0}}\mathcal{F}\overset{[1]}{\rightarrow}.$$ It gives rise to a long
exact sequence in $Perv(X)$
$$\rightarrow{^p}{\! H}{^k}({^p}{\!\tau}{^{\leq -1}}\mathcal{F})\rightarrow
{^p}{\! H}{^k}(\mathcal{F})\rightarrow{^p}{\!
H}{^k}({^p}{\!\tau}{^{\geq 0}}\mathcal{F}) \rightarrow{^p}{\!
H}{^{k+1}}({^p}{\!\tau}{^{\leq -1}}\mathcal{F})\rightarrow. $$ Since
${^p}{\! H}{^*}(\mathcal{F})$ is concentrated in degree $0$,
 we have ${^p}{\! H}{^k}({^p}{\!\tau}{^{\leq -1}}\mathcal{F})=0$ for
all $k\in\mathbb{Z}$. From Lemma \ref{perverse cohomology},
we get $\mathcal{F}$ is isomorphic to ${^p}{\!\tau}{^{\geq 0}}\mathcal{F}$.

Similarly, if we consider the exact triangle
$${^p}{\!\tau}{^{\leq 0}}{^p}{\!\tau}{^{\geq 0}}\mathcal{F}\rightarrow
{^p}{\!\tau}{^{\geq 0}}\mathcal{F}\rightarrow {^p}{\!\tau}{^{\geq
1}}\mathcal{F}\overset{+1}{\rightarrow}$$ and get a long exact
sequence in $Perv(X)$, then we can conclude that
${^p}{\!\tau}{^{\geq 0}}\mathcal{F}$ is isomorphic to
 ${^p}{\!\tau}{^{\leq 0}}{^p}{\!\tau}{^{\geq 0}}\mathcal{F}={^p}{\! H}{^0}(\mathcal{F})$. Therefore,
 $\mathcal{F}\simeq {^p}{\! H}{^0}(\mathcal{F})$ in $D_c^b(X)$.
\end{proof}
Fix a complex stratification $\mathcal{S}=\{S_\alpha\}_{\alpha\in\Lambda}$ of $X$ with
each stratum connected. The perverse
$t$-structure on $D_c^b(X)$ induces the perverse $t$-structure on
$D^b_\mathcal{S}(X)$.

Let $\Lambda_\mathcal{S}:=\bigcup\limits_{\alpha\in\Lambda}T^*_{S_\alpha}X\subset T^*X$ be
the standard conical Lagrangian associated to $\mathcal{S}$. 
For each $S_\alpha\in\mathcal{S}$, let 
$D_{S_\alpha}^*X=T_{S_\alpha}^*X\cap (\bigcup\limits_{\alpha\neq\beta\in\Lambda}\overline{T^*_{S_\beta}X})$. Then the smooth locus in $\Lambda_\mathcal{S}$ is the union 
$\bigcup\limits_{\alpha\in\Lambda}(T_{S_\alpha}^*X-D_{S_\alpha}^*X)$.

\subsection{Local Morse group functor $M_{x,F}$ on $D^b_\mathcal{S}(X)$}\label{LMGF} 
$Perv(X)$ is an abelian subcategory in $D_c^b(X)$. An exact sequence $0\rightarrow \mathcal{F}\rightarrow\mathcal{G}\rightarrow\mathcal{H}\rightarrow 0$ in $Perv(X)$, though corresponds to an exact triangle in $D_c^b(X)$, \emph{does not} give an exact sequence on the stalks. The correct ``stalk" to take in $Perv(X)$, in the sense that it gives an exact sequence, is the microlocal stalk. We now introduce the microlocal stalk under the its another name, local Morse group functor. 

Let $(x,\xi)\in
\Lambda_\mathcal{S}$ be a smooth
point. Fix a local holomorphic coordinate $\mathbf{z}$ around $x$ with origin at $x$, and let $\mathrm{r}(\mathbf{z})=\|\mathbf{z}\|^2$ be the standard distance squared function. Let $F$ be a germ of holomorphic function on $X$, i.e. defined on some small open ball $B_{2\epsilon}(x)=\{\mathbf{z}: \mathrm{r}(\mathbf{z})<(2\epsilon)^2\}\subset X$, such that
$F(x)=0,\ d\Re(F)_x=\xi$ and the graph $\Gamma_{d\Re(F)}$ is transverse to
$\Lambda_\mathcal{S}$ at $(x,\xi)$. We also assume $\epsilon$ small enough so that $x$ is the only $\Lambda_\mathcal{S}$
-critical points of $\Re(F)$. In the following, we will call such a triple $(x,\xi,F)$ a \emph{test triple}.

Let $\phi_{x, F}: D_\mathcal{S}(B_\epsilon(x))\rightarrow
D_c^b(F^{-1}(0)\cap B_\epsilon(x))$ be the vanishing cycle functor
associated to $F$ (see $\S$3 in \cite{Massey} for the definition of nearby and vanishing cycle functors). Note that for any $\mathcal{F}\in
D_\mathcal{S}(X)$, $\phi_{x,F}(\mathcal{F})$ is supported on $x$.

\begin{definition}\label{local Morse group}
Given $(x, \xi, F)$, define the
\emph{local Morse group functor}
$M_{x,F}:=j_{x}^!\phi_{x,F}[-1]l^*= j_{x}^*\phi_{x,F}[-1]l^*:
D^b_\mathcal{S}(X)\rightarrow D^b(\mathbb{C}),$ where $l:
B_\epsilon(x)\hookrightarrow X$ and $j_{x}: \{x\}\hookrightarrow
F^{-1}(0)\cap B_\epsilon(x)$ are the inclusions.
\end{definition}

It's a standard 
fact that on $D^b_\mathcal{S}(X)$, 
\begin{equation}\label{SMT}
M_{x,F}(\mathcal{F})\simeq 
\Gamma(B_\epsilon(x), B_\epsilon(x)\cap F^{-1}(t),
\mathcal{F})\simeq \Gamma(B_\epsilon(x), B_\epsilon(x)\cap \{\Re(F)<\mu\}, \mathcal{F}),
\end{equation} 
where $t$ is any complex number with
$0<|t|<<\epsilon$, and $\mu\leq 0$ with $|\mu|<<\epsilon$.

More generally, let $X$ be a real analytic manifold with a Riemannian metric, and let $\mathcal{S}=\{S_\alpha\}_{\alpha\in\Lambda}$
be a (real) stratification. 
 Assume a function $g: X\rightarrow \mathbb{R}$ satisfies similar conditions of $\Re(F)$ 
at a given point $x\in S_\alpha$, with Morse index of $g|_{S_\alpha}$ equal to $\lambda$. Then given a sheaf $\mathcal{F}$ in $D^b_\mathcal{S}(X)$,
the hypercohomology groups 
\begin{equation}\label{SS(F)}
\mathbb{H}^i(B_\epsilon(x), \{g<0\}\cap B_\epsilon(x),\mathcal{F}[\lambda]), i\in\mathbb{Z}
\end{equation}
are independent of the choices of $g$ and $x$ for $(x,dg_x)$ staying in a fixed connected component
of $T_{S_\alpha}^*X-D_{S_\alpha}^*X$. For more details, see \cite{Massey2} Theorems 2.29, 2.31 and the references therein. In the complex setting, $T_{S_\alpha}^*X-D_{S_\alpha}^*X$ is always connected, so $M_{x,F}(\mathcal{F})$ are quasi-isomorphic for different choices of $(x,\xi)$ in it (but not in a canonical way since there may be monodromies).

One then defines the \emph{singular support} $SS(\mathcal{F})$ of $\mathcal{F}$ as the closure of the covectors in $\Lambda_{\mathcal{S}}$ with the relative hypercohomology groups in $(\ref{SS(F)})$ not all equal to 0.

\begin{lemma}\label{local Morse t-exact}
$M _{x,F}: D^b_\mathcal{S}(X)\rightarrow D^b(\mathbb{C})$ is
$t$-exact. It commutes with the Verdier duality.
\end{lemma}
\begin{proof}
It's standard that $\phi_{x,F}[-1]$ and $l^*$ are perverse
$t$-exact. Since $\phi_{x,F}[-1]l^*(\mathcal{F})$
is supported on $x$, we have
$$M_{x,F}({^p}{\!H}^k(\mathcal{F}))=j^!_{\{x\}}\phi_{x,F}[-1]l^*({^p}{\!H}^k(\mathcal{F}))\simeq
H^k(j^!_{\{x\}}\phi_{x,F}[-1]l^*(\mathcal{F}))=H^k(M_{x,F}(\mathcal{F})).$$
$\phi_{x,F}[-1]$ and $l^*$ commute with $\mathbb{D}$, so it's
easy to see that
 $$M_{x,F}\mathbb{D}=j_{\{x\}}^!\phi_{x,F}[-1]l^*
 \mathbb{D}\simeq \mathbb{D}j_{\{x\}}^*\phi_{x,F}[-1]l^*=\mathbb{D}M_{x,F}.$$
\end{proof}

\begin{lemma}
For each stratum $S_\alpha\in\mathcal{S}$, choose one test triple $(x_\alpha, \xi_\alpha, F_\alpha)$ with $(x_\alpha, \xi_\alpha)\in \Lambda_{S_\alpha}$.  If $M_{x_\alpha,F_\alpha}(\mathcal{F})\simeq 0$ for all $S_\alpha\in\mathcal{S}$, then $\mathcal{F}\simeq
0$.
\end{lemma}
\begin{proof}
From the previous discussion, $M_{x_\alpha,F_\alpha}(\mathcal{F})\simeq 0$ for one choice of $(x_\alpha, \xi_\alpha, F_\alpha)$ is equivalent to $M_{x_\alpha,F_\alpha}(\mathcal{F})\simeq 0$ for all possible choices of $(x_\alpha, \xi_\alpha, F_\alpha)$.

Again, let $\mathcal{S}_{\leq k}$, $0\leq k\leq n=\dim_\mathbb{C} X$, denote the union of all strata in $\mathcal{S}$ of dimension less than or equal to $k$. Let $\mathcal{S}_{>k}=X-\mathcal{S}_{\leq k}$ and $\mathcal{S}_{k}=\mathcal{S}_{\leq k}-\mathcal{S}_{\leq k-1}$. Denote by $i_{k}, i_{>k}, j_{\leq k}$ the inclusion of $S_{\bullet}$ with corresponding subscripts.

For any test triple $(x,0, F)$ with $x\in\mathcal{S}_n$, $x$ is a Morse singularity of $F$
with index 0. By basic Morse theory, 
$M_{x,F}(\mathcal{F})\simeq 0$ implies $i_n^*\mathcal{F}\simeq
0$. By the adjunction exact triangle (\ref{rel_tr}), $\mathcal{F}\simeq
j_{\leq n-1!}j_{\leq n-1}^!\mathcal{F}$.

 In the following we will only look at
$B_\epsilon(x)$, and omit functors related to the open inclusion $l:
B_\epsilon(x)\hookrightarrow X$. Note for $x\in\mathcal{S}_{n-1}$,
by base change formula,
$\phi_{x,F}(j_{\leq n-1!}j_{\leq n-1}^!\mathcal{F})\simeq
\hat{j}_{\leq n-1!}\phi_{x,F_{n-1}}(j_{\leq n-1}^!\mathcal{F})$, where $F_{n-1}$
is the restriction of $F$ to $\mathcal{S}_{n-1}$, and $\hat{j}_{\leq n-1}$
is the inclusion of $F^{-1}(0)\cap \mathcal{S}_{\leq n-1}$ into
$F^{-1}(0)$. Therefore $M_{x,F}(j_{\leq n-1!}j_{\leq n-1}^!\mathcal{F})\simeq
M_{x,F_{n-1}}(j_{\leq n-1}^!\mathcal{F})$. Since $x$ is a Morse singularity of 
$F_{n-1}$ with index 0 on $\mathcal{S}_{n-1}$, by previous argument,
$j_{n-1}^! \mathcal{F}\simeq 0$. By Verdier duality, 
$j_{n-1}^* \mathcal{F}\simeq 0$ as well. Applying the adjunction
exact triangle again to the open set $\mathcal{S}_{\geq n-1}$, we get
$\mathcal{F}\simeq j_{\leq n-2!}j^!_{\leq n-2}\mathcal{F}$, and by
induction, we get $\mathcal{F}\simeq 0$.
\end{proof}

Combining the two lemmas, we immediately get the following.

\begin{prop}[Microlocal characterization of perverse sheaves] \label{M_{x,F}-Test}

For each stratum $S_\alpha\in\mathcal{S}$, choose a test triple $(x_\alpha, \xi_\alpha, F_\alpha)$ with $(x_\alpha, \xi_\alpha)\in \Lambda_{S_\alpha}$.
Then $\mathcal{F}\in D_\mathcal{S}^b(X)$ is perverse if and only 
$M_{x_\alpha,F_\alpha}(\mathcal{F})$ has cohomology group concentrated in degree $0$ for all $S_\alpha\in\mathcal{S}$.
\end{prop}

\subsection{$M_{x,F}$ as a functor on the dg-category $Sh_\mathcal{S}(X)$}
$M_{x,F}$ can be natually viewed as a dg-functor from $Sh_\mathcal{S}(X)$ to $Ch$, where $Ch$ denotes the dg-category of cochain complexes of vector spaces. And we have a natual identification
$$M_{x,F}\simeq\Gamma(B_\epsilon(x), B_\epsilon(x)\cap \{\Re(F)<\mu\}, -)$$ 
for sufficiently
small $\epsilon>0$ and $-\epsilon<<\mu\leq 0$. 

To make future calculations easier, we refine $\mathcal{S}$ into a new (real) stratification $\widetilde{\mathcal{S}}$ with each stratum a cell, and view $Sh_\mathcal{S}(X)$ as a subcategory of $Sh_{\widetilde{\mathcal{S}}}(X)$. $M_{x,F}$ is obviously extended to $Sh_{\widetilde{\mathcal{S}}}(X)$, as long as $x$ is not lying in any newly added stratum, and the   microlocal characterization for perverse sheaves (Proposition \ref{M_{x,F}-Test}) still applies for $Sh_\mathcal{S}(X)$. In the following, to simplify notation, we still denote $\widetilde{\mathcal{S}}$ by $\mathcal{S}$.

We have seen in Section \ref{sheaves} that $Sh_{\mathcal{S}}(X)$ is generated by $i_*\mathbb{C}_U$ for $U\in \mathcal{U}_\mathcal{S}=\{X, O_\alpha=X-\overline{S_\alpha}, O'_\alpha=X-\partial S_{\alpha}:  S_{\alpha}\in \mathcal{S}\}$. So to understand $M_{x,F}$, it suffices to understand its interaction with these standard generators. It is easy to see that $M_{x,F}$ is only nontrivial on the finite subcolletion of $i_*\mathbb{C}_U$, where 
\begin{equation}\label{U_S_x}
U\in \mathcal{U}_{\mathcal{S}, x}:=\{V\in\mathcal{U}_{\mathcal{S}}: x\in\overline{V}\}.
\end{equation}

\begin{lemma}\label{wedge}
For each $i_V: V\hookrightarrow X$ open, consider the dg functor $\Gamma(V,-): Sh(X)\rightarrow Ch$.
For any two open embeddings
$i_0:U_0\hookrightarrow X, i_1: U_1\hookrightarrow X$, the
composition map
\begin{equation}\label{hom}
\text{Hom}_{\text{Sh}(X)}(i_{0*}\mathbb{C}_{U_0},i_{1*}\mathbb{C}_{U_1})\otimes \Gamma(V,
i_{0*}\mathbb{C}_{U_0})\rightarrow
\Gamma(V, i_{1*}\mathbb{C}_{U_1})
\end{equation}
is canonically identified with the
wedge product on the deRham complexes
$$(\Omega (\overline{U}_0\cap U_1, \partial U_0\cap U_1), d)\otimes (\Omega(U_0\cap V), d) \rightarrow (\Omega(U_1\cap V), d).$$
\end{lemma}
\begin{proof}
The map (\ref{hom}) is coming from the induced map on (derived) global sections of the following sheaf map
\begin{equation}\label{sheaf hom}
 i_{1*}i_1^*i_{0!}\mathbb{C}_{U_0}\otimes i_{V*}i_V^*i_{0*}\mathbb{C}_{U_0} \rightarrow
i_{V*}i_V^*i_{1*}\mathbb{C}_{U_1}.
\end{equation}
Resolve the sheaves by $c$-soft sheaves (see Definition 2.5.5 in \cite{Kashiwara}) as $i_{V*}i_V^*i_{0*}\mathbb{C}_{U_0}\simeq \Omega^{\bullet}_{V\cap U_0}$, 
$i_{1*}i_1^*i_{0!}\mathbb{C}_{U_0}\simeq \Omega^{\bullet}_{\overline{U}_0\cap U_1,\partial U_0\cap U_1}$, and $i_{V*}i_V^*i_{1*}\mathbb{C}_{U_1}\simeq \Omega^{\bullet}_{V\cap U_1}$, where $\Omega^{\bullet}_A$ of a set $A$ means the sheaf of smooth differential forms on $A$, and $\Omega^{\bullet}_{A,B}$ means the subsheaf of $\Omega^{\bullet}_A$, whose local sections have support away from $B$. The map (\ref{sheaf hom}) is the wedge product on differential forms, which induces wedge product on the global sections. \\
\end{proof}

\begin{cor}
(1) The functor $M_{x,F}(-)$ on $Sh_{\mathcal{S}}(X)$ fits into the
exact triangle $ M_{x,F}(-)\longrightarrow
\Gamma(B_\epsilon(x),-)\longrightarrow\Gamma(B_\epsilon(x)\cap
\{\Re(F)<0\},-)\overset{[1]}{\longrightarrow}$. \\
(2) Given $U_0, U_1$ open in $X$, the composition map 
$$
Hom_{Sh(X)}(i_{0*}\mathbb{C}_{U_0},i_{1*}\mathbb{C}_{U_1})\otimes M_{x,F}(i_{0*}\mathbb{C}_{U_0})\rightarrow
M_{x,F}(i_{1*}\mathbb{C}_{U_1})$$ is canonically given by the wedge
product on the deRham complexes:
\begin{align*}
&(\Omega(\overline{U_0}\cap U_1,\partial
U_0\cap U_1),d)\otimes (\Omega(U_0\cap B_\epsilon(x), U_0\cap B_\epsilon(x)\cap\{\Re(F)<0\}),d)\\
&\rightarrow (\Omega(U_1\cap B_\epsilon(x), U_1\cap B_\epsilon(x)\cap
\{\Re(F)<0\}),d).
\end{align*}
\end{cor}

\section{The Nadler-Zaslow Correspondence}\label{NZCorr}

\subsection{Two categories: $Open(X)$ and $Mor(X)$}\label{OpenMor}
Let $Open(X)$ be the dg category whose objects are open subsets in $X$ with a semi-defining function (see Remark \ref{def fcn}). For two objects $\mathfrak{U}_0=(U_0, m_0), \mathfrak{U}_1=(U_1,m_1)$, define
$$Hom_{Open(X)}(\mathfrak{U}_0, \mathfrak{U}_1):=Hom(i_{0*}\mathbb{C}_{U_0}, i_{1*}\mathbb{C}_{U_1})\simeq (\Omega(\overline{U}_0\cap U_1, \partial U_0\cap U_1), d).$$
The composition for a triple is the wedge product on deRham complexes as in Proposition \ref{open_hom}. 
From previous discussions, $Sh(X)$ is a triangulated envelope of $Open(X)$. 

Define another $A_\infty$-category, denoted as $Mor(X)$ with the same objects as $Open(X)$. The morphism between two objects $\mathfrak{U}_i=(U_i,m_i), i=1,2$ is defined by the Morse complex calculation of $Hom(i_{0*}\mathbb{C}_{U_0}, i_{1*}\mathbb{C}_{U_1})$, using perturbation to smooth transverse boundaries similar to the process at the end of Section \ref{sheaves}. Let $f_i=\log m_i$ for $i=0,1$. Pick a stratification $\mathcal{T}$ compatible with $\partial U_0$. There is $\bar{t}_1>0$ such that $m_1$ has no $\Lambda_\mathcal{T}$-critical value in $(0,\bar{t}_1)$. Fix $t_1\in(0,\bar{t}_1)$. Let $W$ be a small neighborhood of $\partial U_0\cap X_{m_1=t_1}$ on which $df_0$ and $df_1$ are linearly independent. Since $X_{m_1=t_1}\cap U_0-W$ is compact, one could dilate $df_0$ by $\epsilon>0$ so that $|df_1|>2\epsilon |df_0|$ on $X_{m_1=t_1}\cap U_0-W$. There is $\bar{t}_0>0$ such that for any $t\in (0,\bar{t}_0)$, $X_{m_0=t}$ intersects $X_{m_1=t_1}$ transversally. Choose $t_0\in (0,\bar{t}_0)$ such that $|\epsilon\cdot df_0|>2|df_1|$ on $X_{m_0=t_0}\cap X_{m_1\geq t_1}-W$. Such a $t_0$ always exists, since $df_1$ is bounded on $X_{m_1\geq t_1}$. There is also a convex space of choices of Riemannian metric $g$ in a neighborhood of $X_{m_0\geq t_0}\cap X_{m_1\geq t_1}$, with respect to which the gradiant vector field $\nabla (f_1-\epsilon f_0)$ is pointing outward along $X_{m_0=t_0}\cap X_{m_1\geq t_1}$, and inward along $X_{m_0\geq t_0}\cap X_{m_1= t_1}$. After small perturbations, one can perturb the function $f_1-\epsilon f_0$ to be Morse, and the pair $(f_1-\epsilon f_0, g)$ to be Morse-Smale. 

Let $M$ be an $n$-dimensional \emph{manifold with corners}. By definition, for every point $y\in\partial M$, there is a local chart $\phi_y: U_y\rightarrow \mathbb{R}^n$ identifying an open neighborhood $U_y$ of $y$ with an open subset of a quadrant $\{x_{i_1}\geq 0,...,x_{i_k}\geq 0\}$. We will say a function $f$ on $X$ is \emph{directed}, if (1) $\phi_{y*}f$ can be extended to be a smooth function on an open neighborhood of $\phi_y(U_y)$, and with respect to some Riemannian metric $g$, the gradient vector field of the resulting function is pointing either strictly outward or strictly inward along every face of  $\phi_y(U_y)$, (2) $f$ is a Morse function on $M$ and the pair $(f, g)$ is Morse-Smale. We will also call $(f,g)$ a \emph{directed pair}. From the above discussion, $f_1-\epsilon f_0$ is directed on the manifold with corners $X_{m_0\geq t_0}\cap X_{m_1\geq t_1}$.

Now define
$$Hom_{Mor(X)}(\mathfrak{U}_0, \mathfrak{U}_1):= Mor^*(X_{m_0> t_0}\cap X_{m_1>t_1}, f_1-\epsilon f_0), $$ 
where $Mor^*(X_{m_0>t_0}\cap X_{m_1>t_1}, f_1-\epsilon f_0)$ is the usual Morse complex associated to the function $f_1-\epsilon f_0$ (after small perturbations when necessary).  
It is clear from the above description that the definition essentially doesn't depend on the choices of $t_0, t_1, \epsilon$ and $g$. There are compatible quasi-isomorphisms between the complexes with different choices.

The (higher) compositions are defined by counting Morse trees as follows.  
\begin{definition}
A \emph{based metric Ribbon tree} $T$ is an embedded tree into the unit disc consisting of the following data. \\
1. Vertices: there are $n+1$ points on the boundary of the unit disc in $\mathbb{R}^2$ labeled counterclockwise by $v_0,...,v_n$, where $v_0$ is referred as the \emph{root vertex}, and others are referred as \emph{leaf vertices}. There is a finite set of points in the interior of the disc, which are referred as \emph{interior vertices}. \\
2. Edges: there are straight line segments referred as \emph{edges} connecting the vertices. An edge $e$ connecting to the root or a leaf is called an \emph{exterior} edge, otherwise it is called an \emph{interior} edge. We will use $e_i$ to denote the unique exterior edge attaching to $v_i$, and $e_{in}$ to denote an interior edge. The resulting graph of vertices and edges should be a connected embedded stable tree in the usual sense, i.e. the edges do not intersect each other in the interior, there are no cycles in the graph, and each interior vertex has at least 3 edges.\\
3. Metric and orientation: the tree is oriented from the leaves to the root, in the direction of the shortest path (measured by the number of passing edges). Each interior edge $e_{in}$ is given a \emph{length} $\lambda(e_{in})>0$. One could parametrize the edges as follows, but the parametrization is not part of the data. Each $e_{in}$ is parametrized by the bounded interval $[0, \lambda(e_{in})]$ respecting the orientation. Every $e_i-\{v_i\}, i\neq 0$ is parametrized by $(-\infty, 0]$ and $e_0-\{v_0\}$ is parametrized by $[0,\infty)$.\\
4. Equivalence relation: two based metric Ribbon trees are considered the same if there is an isotopy of the closed unit disc which identifies the above data.
\end{definition}

Let $\mathfrak{U}_i=(U_i,f_i), i\in\mathbb{Z}/(k+1)\mathbb{Z}$ be a sequence of objects in $Mor(X)$ (when we compare the magnitude of two indices, we think of them as natural numbers ranging from $0$ to $k$). We can apply the perturbation process as before to produce a \emph{directed sequence} $\tilde{\mathfrak{U}}_i=(\tilde{U}_i,\tilde{f}_i)$, where $\partial \tilde{U}_i$'s are all smooth and transversely intersect with each other, and $\tilde{f}_j-\tilde{f}_i$ is directed on $\overline{\tilde{U}}_i\cap \overline{\tilde{U}}_j$ for $j>i$ (the boundary on which $\nabla(\tilde{f}_j-\tilde{f}_i)$ is pointing outward is understood to be $\partial \tilde{U}_i\cap \overline{\tilde{U}}_j$). A \emph{Morse tree} is a continuous map $\phi: T\rightarrow X$ such that\\
1. $\phi(v_i)\in Cr(\tilde{U}_{i-1}\cap \tilde{U}_i, \tilde{f}_{i}-\tilde{f}_{i-1})$, for $i\in\mathbb{Z}/(k+1)\mathbb{Z}$;\\
2. The tree divides the disc into several connected components, and we label these components counterclockwise starting from $0$ on the left-hand-side of $e_0$ (with respect to the given orientation). Let $\ell(e)$ and $r(e)$ denote the label on the left and right-hand-side of an edge $e$ respectivley. Then we require that $\phi|_{e}$ is $C^1$ and under some parametrization of the edges, we have
\begin{align*}
&\frac{d\phi(t)}{dt}|_{e_{in}}=\nabla(\tilde{f}_{\ell(e_{in})}-\tilde{f}_{r(e_{in})}), \text{ for }t\in(0,\lambda(e_{in})),\\
&\frac{d\phi(t)}{dt}|_{e_i}=\nabla(\tilde{f}_{\ell(e_i)}-\tilde{f}_{r(e_i)}), \text{ for }t\in(-\infty, 0) \text{ if }i\neq 0 \text{ and }t\in (0,\infty) \text{ if }i=0.
\end{align*}

For $a_i\in Cr(\tilde{U}_{i}\cap\tilde{U}_{i+1}, \tilde{f}_{i+1}-\tilde{f}_{i}), i\in\mathbb{Z}/(k+1)\mathbb{Z}$, let $\mathcal{M}(T; \tilde{f}_0,...,\tilde{f}_{k-1}; a_0,...,a_{k})$ denote the moduli space of Morse trees with $\phi(v_i)=a_{i-1}$. After a small perturbation of the functions, this moduli space is regular, and the signed count of the 0-dimensional part $\mathcal{M}(T; \tilde{f}_0,...,\tilde{f}_{k}; a_0,...,a_{k})^{0\text{-d}}$ gives the higher compositions
$$m^k_{Mor(X)}(a_{k-1}, a_{k-2},...,a_0)=\sum\limits_{b_k\in Cr(\tilde{U}_0\cap\tilde{U}_k, \tilde{f}_{k}-\tilde{f}_{0})}\sharp\mathcal{M}(T; \tilde{f}_0,...,\tilde{f}_{k}; a_0,...,a_{k-1}, b_{k})^{0\text{-d}} \cdot b_k.$$

\subsection{$Open(X)\simeq Mor(X)$ via Homological Perturbation Lemma }\label{Op.eq.Mor}

Let's recall the Homological Perturbation Lemma summarized in \cite{Seidel}. Assume we are given an $A_\infty$-category $\mathcal{A}$ and a
collection of chain maps $F, G$ on $Hom_\mathcal{A}(X_1,X_2)$ for
each pair of objects $X_1,X_2$ such that 
$$G\circ F=\text{Id},\
F\circ G-\text{Id}=m_\mathcal{A}^1\circ H+H\circ m_\mathcal{A}^1,$$ where $H$ is a map on
$Hom_\mathcal{A}(X_1,X_2)$ of degree $-1$.

\begin{thm}
There exists an $A_\infty$-category $\mathcal{B}$ with the same
objects, morphism spaces  and $m^1$ as $\mathcal{A}$. This comes with $A_\infty$-morphisms
$\mathcal{F}: \mathcal{B}\rightarrow \mathcal{A},\ \mathcal{G}:\mathcal{A}\rightarrow \mathcal{B}$ which are identity on objects and $\mathcal{F}^1=F,
\mathcal{G}^1=G$. There is also a homotopy $\mathcal{H}$ between $\mathcal{F}\circ\mathcal{G}$ and $Id_\mathcal{A}$ such that $\mathcal{H}^1=H$. 
\end{thm}

\begin{remark}\label{HPL}
In this paper, all $A_\infty$-categories and $A_\infty$-functors are assumed to be $c$-unital. The Homological Perturbation Lemma generalizes to
left $A_\infty$-modules, namely, in addition to the above data, let
$\mathcal{M}: \mathcal{A}\rightarrow Ch$ be a left
$A_\infty$-module over $\mathcal{A}$ and assume for each object $X$, there
are chain maps $\tilde{F}, \tilde{G}$ and a homotopy $\tilde{H}$ on $\mathcal{M}(X)$ satisfying
$$\tilde{G}\circ \tilde{F}=\text{Id},\
\tilde{F}\circ \tilde{G}-\text{Id}=d\circ\tilde{ H}+\tilde{H}\circ d,$$ 
then one can construct a left $A_\infty$-module $\mathcal{N}$ over $\mathcal{B}$ such that
$$\mathcal{N}^1=\tilde{G}\circ \mathcal{M}^1\circ F,$$
and there are module homomorphisms
$$t: \mathcal{N}\rightarrow \mathcal{F}^*\mathcal{M},\  s: \mathcal{M}\rightarrow\mathcal{G}^*\mathcal{N}.$$
Then we have the following composition
$$T=(\mathcal{R}_\mathcal{G}(t))\circ s: \mathcal{M}\rightarrow\mathcal{N}
\circ\mathcal{G}\rightarrow\mathcal{M}\circ\mathcal{F}\circ\mathcal{G},$$
where $\mathcal{R}_\mathcal{G}$ is taking composition with $\mathcal{G}$ on the right.
Using the homotopy between $\mathcal{F}\circ\mathcal{G}$ and $Id_\mathcal{A}$, we get a composition of morphisms between the induced cohomological functors
$$H(T): H(\mathcal{M})\xrightarrow{H(s)} H(\mathcal{N}
\circ\mathcal{G})\xrightarrow{H(\mathcal{R}_{\mathcal{G}}(t))} H(\mathcal{M})=H(\mathcal{M}\circ\mathcal{F}\circ\mathcal{G}).$$
Then it is easy to check that $H(s)\circ H(\mathcal{R}_{\mathcal{G}}(t))=id$ and $H(T)$ is an isomorphism, so 
$$H(s): H(\mathcal{M})\xrightarrow{\sim}H(\mathcal{G}^*\mathcal{N}).$$
\end{remark}

Here we will use the version where $G$ is an idempotent, $Hom_{\mathcal{B}}(X_1,X_2)$ is the
image of $G$ and $F: Hom_{\mathcal{B}}(X_1,X_2)\rightarrow Hom_{\mathcal{A}}(X_1,X_2)$ is the inclusion; see \cite{KS}. The same for $\tilde{G}$ and $\tilde{F}$.

Let $(X, g)$ be a Riemannian manifold with corners. Let
$(f,g)$ be a directed pair with $\varphi_t$ the
gradient flow of $f$. Denote by $H_0\subset \partial X$ the
hypersurface where $\nabla f$ is pointing outward, and
$H_1\subset\partial X$ the hypersurface where $\nabla f$ is
pointing inward. Let $D'(X-H_0,H_1), D'(X-H_1,H_0)$, called \emph{relative currents}, be the dual of
$\Omega(X-H_1,H_0)$ and $\Omega(X-H_0,H_1)$ respectively.

In the following, we briefly recall the idempotent functor on $\Omega(X, H_0)$ constructed in \cite{HL} and \cite{KS}
and used by \cite{NZ} in the manifold-with-corners setting. 
Consider the functor
\begin{align}\label{Proj}
P: \Omega(X-H_1, H_0)&\rightarrow D'(X-H_1,
H_0)\\
\alpha &\mapsto \sum\limits_{x\in Cr(f)}(\int_{\mathscr{U}_x}\alpha)[\mathscr{S}_x],\nonumber
\end{align}
where $Cr(f)$ is the set of critical points of $f$, $\mathscr{U}_x$ is the unstable manifold associated to $x$ and $\mathscr{S}_x$
is the stable manifold associated to $x$. 

There is a homotopy functor $T$ between $P$ and
the inclusion $I: \Omega(X-H_1, H_0)\hookrightarrow D'(X-H_1,
H_0)$ given by the current $\{(\varphi_t(y), y): t\in \mathbb{R}_{\geq 0}\}\subset
X\times X$ in $D'(X\times X)$. To construct a real idempotent functor on $\Omega(X-H_1,H_0)$, one composes it with a smoothing
functor. For readers interested in further details, see \cite{KS}.

A consequence of the functor $P$ is the Morse theory for manifolds with corners:
\begin{equation}\label{MTMC}
\Omega(X-H_1, H_0)\simeq Mor^*(X, f).
\end{equation}
Following the notations in Section \ref{OpenMor}, this implies the following canonical quasi-isomorphisms
\begin{align*} 
&Hom_{Mor(X)}(\mathfrak{U}_0, \mathfrak{U}_1)\simeq Mor^*(X_{m_0>t_0}\cap X_{m_1>t_1}, f_1-\epsilon f_0)\\
\simeq& (\Omega(X_{m_0\geq t_0}\cap X_{m_1>t_1}, X_{m_0=t_0}\cap X_{m_1>t_1}), d)\simeq  (\Omega(\overline{U}_0\cap U_1, \partial U_0\cap U_1), d)\\
=&Hom_{Open(X)}(\mathfrak{U}_0, \mathfrak{U}_1).
\end{align*}

Applying Homological Perturbation Lemma to the dg category $Open(X)$ through the functors $P,\ I$ and $H$ for each pair of objects, one can show that $Mor(X)$ is exactly the $A_\infty$-category $\mathcal{B}$ in Theorem \ref{HPL} constructed out of these data. Using the set-up for defining higher morphisms in $Mor(X)$ of Section \ref{OpenMor}, there is a nice description of $\mathcal{M}(T; \tilde{f}_0,...,\tilde{f}_{k-1}; a_0,...,a_{k})$ in terms of intersections of the stable manifold of $a_i$ for  $i\neq k$ and the unstable manifold of $a_k$, which also involves the functors $P$ and $I$. After a smoothing functor, one replaces intersection of currents by wedge product on differential forms, then compare this with the formalism of Homological Perturbation Lemma to get the assertion. For more details about the argument, see \cite{KS}. We will use the same idea in Section \ref{simeq} for left $A_\infty$-modules.  

\subsection{The microlocalization $\mu_X: Sh(X)\xrightarrow\sim Tw Fuk(T^*X)$ }\label{microlocal}
For any $\mathfrak{U}=(U, m)\in Mor(X)$, we can associate the standard brane $L_{U,m}$ in $Fuk(T^*X)$ (see Appendix \ref{basT*X} (b)), and in this way $Mor(X)$ is naturally identified with the $A_\infty$-subcategory of $Fuk(T^*X)$ generated by these standard branes. 

Roughly speaking, one does a series of appropriate perturbations and dilations to the branes $L_{U_i, m_i}$, $i=1,...,k$, so that 
(1) after further variable dilations (see Appendix \ref{basT*X} (c)), one can use the monotonicity properties (Proposition \ref{EB}, Remark \ref{U.E.B}) to get that all holomorphic discs bounding the (dilating family of) the new branes $\epsilon\cdot L_i, i=1,...,k$ have boundary lying in the partial graphs $\epsilon\cdot L_i|_{\tilde{U}_i}=\epsilon\cdot \Gamma_{d\tilde{f}_i}, i=1,...,k$, where $\tilde{U}_i$ is a small perturbation of $U_i$ and $\tilde{f}_i: \tilde{U}_i\rightarrow \mathbb{R}$ is some function; (2) the sequence $(\tilde{U}_i, \tilde{f}_i), i=1,...,k$ is directed, and hence one could use $(\tilde{U}_i, \tilde{f}_i)$ as representatives in the calculation of (higher) morphisms involving $\mathfrak{U}_i=(U_i, m_i), i=1,...,k$ in $Mor(X)$. Since we will use the same technique in Section \ref{df, L_V} and \ref{simeq}, we defer the details there.

Recall Fukaya-Oh's theorem. 
\begin{thm}[\cite{Fukaya}]\label{FOthm}
For a compact Riemannian manifold $(X,g)$, and a generic sequence of functions $f_1,...,f_k$ on $X$, there is an orientation preserving diffeomorphism between moduli space of holomorphic discs (with respec to the Sasaki almost complex structure) bounding the sequence of graphs $\epsilon\cdot \Gamma_{df_i}, i=1,...,k$ and the moduli space of Morse trees for the sequence $(X,\epsilon f_i), i=1,...,k$, for all $\epsilon>0$ sufficiently small. 
\end{thm}
Since the proof of the theorem is local and essentially relies on the $C^1$-closeness of the graphs to the zero section, one could adapt it to the directed sequence $(\tilde{U}_i, \tilde{f}_i), i=1,...,k$ and conclude that the moduli space of discs bounding $\epsilon\cdot L_i, i=1,...,k$ is diffeomorphic (as oriented manifolds) to the moduli space of Morse trees for the sequence $(\tilde{U}_i, \tilde{f}_i), i=1,...,k$. Therefore we get the quasi-embedding $i: Mor(X)\hookrightarrow Fuk(T^*X)$.

Next one composes $i$ with the quasi-equivalence $\mathcal{P}: Open(X)\rightarrow Mor(X)$ from Section \ref{Op.eq.Mor}, and get a quasi-embedding $i\circ\mathcal{P}: Open(X)\rightarrow Fuk(T^*X)$. Then taking twisted complexes on both sides, we get the microlocal functor $\mu_X: Sh(X)\rightarrow Tw Fuk(T^*X)$. To simplify notation, we will denote $Tw Fuk(T^*X)$ by $F(T^*X)$. The main idea in \cite{Nadler2} of proving that $\mu_X$ is a quasi-equivalence is to resolve the conormal to the diagonal in $T^*(X\times X)$ using product of standard branes in $T^*X$. Since we will only use the statement, we refer interested readers to \cite{Nadler2} for details.

For a fixed stratification $\mathcal{S}$, let $Fuk_{\mathcal{S}}(T^*X)$ be the full subcategory of $Fuk(T^*X)$ consisting of branes $L$ with $L^\infty\subset T^\infty_\mathcal{S}X$, and let $F_\mathcal{S}(T^*X)$ denote its twisted complexes. Then we also have 
\begin{equation}\label{micro_S}
\mu_X|_{Sh_\mathcal{S}(X)}: Sh_{\mathcal{S}}(X)\overset{\sim}{\rightarrow} F_{\mathcal{S}}(T^*X). 
\end{equation}

\section{Quasi-representing $M_{x,F}$ on $Fuk_\mathcal{S}(T^*X)$ by the local Morse brane $L_{x,F}$}\label{Repr M_{x,F}}

Continuing the convention from Section \ref{LMG}, for a complex stratification $\mathcal{S}$, we refine it to have each stratum a cell, and denote the resulting stratification by $\mathcal{S}$ as well. The test triples $(x,\xi, F)$ we are considering for $\Lambda_\mathcal{S}$ are always away from the newly added strata. 

Given a test triple $(x,\xi, F)$, we will construct a Lagrangian brane $L_{x,F}$
supported on a neighborhood of $x$, such that the functor $Hom_{F(T^*X)}(L_{x,F},-): F_\mathcal{S}(T^*X)
\rightarrow Ch$ under pull-back by $\mu_X$ is quasi-isomorphic to the local Morse group functor $M_{x,F}$. 
 
\subsection{Construction of $L_{x,F}$}\label{constr}
Consider the function $\text{r}\times \Re(F):
B_{2\epsilon}(x)\rightarrow\mathbb{R}^2$, where $\mathrm{r}(\mathbf{z})=\|\mathbf{z}\|^2$ as before. Let
$R$ be an open subset of $\Lambda_\mathcal{S}\text{-regular values of }\text{r}\times\Re(F)$ in $\mathbb{R}^2$, such that it contains $(0, \delta)\times\{0\}$ for some $\delta>0$, and if $(a,b), (a,c)\in R$ for $b<c$, then $\{a\}\times [b,c]\subset R$ (here we have used Lemma \ref{critical}). There exists a $0<\tilde{r}_2<\delta$ for which the function $\mathrm{r}$ has no $\Lambda_\mathcal{S}$-critical value in $(0,\tilde{r}_2)$ and the following Lemma \ref{r'_2} holds. Fixing such a $\tilde{r}_2$, choose
$0<\tilde{r}_1<\tilde{r}_2$ and $\eta>0$ small enough so that $R$ contains $(\tilde{r}_1,\tilde{r}_2)\times
(-2\eta,2\eta)$, and $\Re(F)$ has no $\Lambda_\mathcal{S}$-critical value in $(-2\eta, 0)$ or $(0,2\eta)$. Also choose $\tilde{r}_1<r_1<r_2<\tilde{r}_2$.

\begin{lemma}\label{r'_2}
There exists $0<\tilde{r}_2<\delta$ such that if $d\Re(F)\in \lambda d\text{r}+\Lambda_\mathcal{S}$ at a point $y\in B_{\tilde{r}_2}(x)\cap \{\Re(F)>0\}$, then $\lambda>0$.
\end{lemma}
\begin{proof}
Suppose the contrary, there is a sequence of points $y_n$ approaching to $x$ such that $d\Re(F)\in\lambda_n d\text{r}+\Lambda_\mathcal{S}$ at $y_n$ for $\lambda_n<0$ (note that $d\Re(F)\notin\Lambda_\mathcal{S}$ except at $x$). Since the stratification is locally finite, we can assume that those points are in a fixed stratum $S_\alpha$ with $x\in\overline{S}_\alpha$. Then by curve selection lemma (Proposition \ref{CSL}), there is a $\mathcal{C}$-curve of $C^1$-class $\rho: [0,1)\rightarrow X$ with $\rho(0)=x$, $\rho((0,1))\subset \{\Re(F)>0\}\cap S_\alpha$, and $\frac{d\Re(F)(\rho(t))}{dt}=\lambda(t)\frac{d\text{r}(\rho(t))}{dt}<0$, which is impossible.
\end{proof}

Let $U^{pre}=\{|\Re(F)|<\eta\}\cap B_{r_2}(x)$. 
Also let 
$$\mu=-\frac{1}{2}\eta, \delta_1=\frac{1}{2}(r_2-r_1), \delta_2=\frac{1}{4}\eta.$$
and 
$$u(\mathbf{z})=\mathrm{r}(\mathbf{z})-(r_2-\delta_1), v(\mathbf{z})=\Re(F)(\mathbf{z})-(\mu-\delta_2).$$
Near $\{u=v=0\}$, we smooth the corners in  
$$W_1:=\{u=0, v\leq 0\}\cup \{u\leq 0, v=0\}$$
as follows. Let $\bar{\epsilon}_1=\frac{1}{2}\min(\delta_1, \delta_2, 1)$. We remove the portion $\{u^2+v^2\leq \bar{\epsilon}_1^2\}$ from $W_1$ and glue in $3/4$ of the cylinder $\{u^2+v^2=\bar{\epsilon}_1^2\}$, i.e. the part where $u, v$ are not both negative. Then we smooth the connecting region so that its (outward) unit conormal vector is always a linear combination of $d\mathrm{r}$ and $d\Re(F)$, in which at least one of the coefficients is positive. This can be achieved by looking at the local picture in the leftmost corner of Figure \ref{figure}, where we complete $u,v$ to be the coordinates of a local chart. We will denote the resulting hypersurface by $\widetilde{W}_1$.

Now we choose a semi-defining function $m_{\widetilde{W}_1}$ for $\widetilde{W}_1$ such that \\
(i) in an open neighborhood $\mathcal{U}_1$ of $\widetilde{W}_1$, $m_{\widetilde{W}_1}$ is a function of $u,v$, and  $dm_{\widetilde{W}_1}\neq 0$,\\
(ii) $m_{\widetilde{W}_1}=u$ on $\{v\leq -\frac{4}{3}\bar{\epsilon}_1, |u|\leq \frac{1}{2}\bar{\epsilon}_1\}$, $m_{\widetilde{W}_1}=v$ on $\{u\leq-\frac{4}{3}\bar{\epsilon}_1, |v|\leq \frac{1}{2}\bar{\epsilon}_1\}$. 

Then there exists $0<\epsilon_1<\frac{1}{2}\bar{\epsilon}_1$ so that $\{0\leq m_{\widetilde{W}}\leq \epsilon_1\}$ is contained in $\mathcal{U}_1$, and over it $dm_{\widetilde{W}_1}$ is a linear combination of $d\mathrm{r}$ and $d\Re(F)$ in which at least one of the coefficients is positive.

The upshot is that the smoothing is within a $\Lambda_\mathcal{S}$-regular locus of $\mathrm{r}\times\Re(F)$, and $\widetilde{W}_1$ and $m_{\widetilde{W}_1}$ are chosen in a way that any positive linear combination of $dm_{\widetilde{W}_1}$, $d\mathrm{r}$ and $d\Re(F)$ will never be zero on $\{0\leq m_{\widetilde{W}_1}\leq \epsilon_1\}$, and hence will not lie in $\Lambda_\mathcal{S}$.

Similarly, we can smooth the corner in $\{\mathrm{r}=r_2, \Re(F)\leq \eta\}\cup \{\mathrm{r}\leq r_2, \Re(F)=\eta\}$, but in the way illustrated on the right hand side of Figure \ref{figure}. We will denote the resulting hypersurface $\widetilde{W}_0$. We choose a semi-defining function $m_{\widetilde{W}_0}$ and $\epsilon_0>0$ in the same fashion as for $m_{\widetilde{W}_1}$ and $\epsilon_1$, making sure that any positive linear combination of $dm_{\widetilde{W}_0}$, $d\mathrm{r}$ and $d\Re(F)$ will never lie in $\Lambda_\mathcal{S}$ over the region $\{-\epsilon_0\leq m_{\widetilde{W}_0}\leq 0\}$.

Let $b:
(-\infty,\eta)\rightarrow\mathbb{R}$ be a nondecreasing $C^1$-function
such that $b(x)=x \text{ for }x\in (\frac{-\eta}{4},\frac{\eta}{2})$, $\lim\limits_{x\rightarrow\eta^-}b(x)=+\infty$, and 
the derivative $b'=0$ exactly on $(-\infty, \mu-\delta_2]$. Let $c:(0, r_2)\rightarrow\mathbb{R}_{\geq 0}$ be a nondecreasing $C^1$-function
such that $\lim\limits_{x\rightarrow
r_2^-}c(x)=+\infty$, and $c'=0$ exactly on $(0, r_2-\delta_1]$. Let $d: (-\infty, 0)\rightarrow \mathbb{R}_{\geq 0}$ be a nondecreasing $C^1$-function with $\lim\limits_{x\rightarrow 0^-}d(x)=+\infty$, and $d'=0$ exactly on $(-\infty, -\epsilon_0]$. Let
$e: (0,+\infty)\rightarrow \mathbb{R}_{\leq 0}$ be a nondecreasing $C^1$-function with $\lim\limits_{x\rightarrow 0^+}e(x)=-\infty$ and 
$e'=0$ exactly on $[\epsilon_1,+\infty)$.

Now define
$$f=b\circ\Re(F)+c\circ \text{r}+d\circ m_{\widetilde{W}_0}+e\circ m_{\widetilde{W}_1}$$ on 
$$U=\text{ the domain bounded by the }\widetilde{W}_0\text{ and }\widetilde{W}_1.$$ 
The construction of $U$ is interpreted in Figure \ref{figure}, where $U$ is the shaded area.

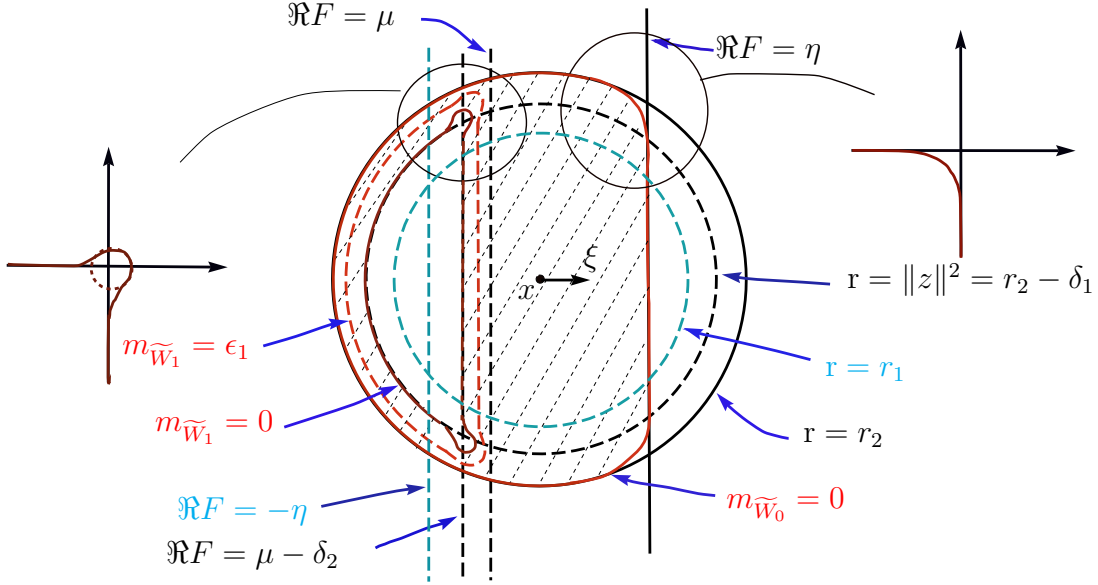
\begin{figure}
\centering
  \begin{overpic}[height=8cm]{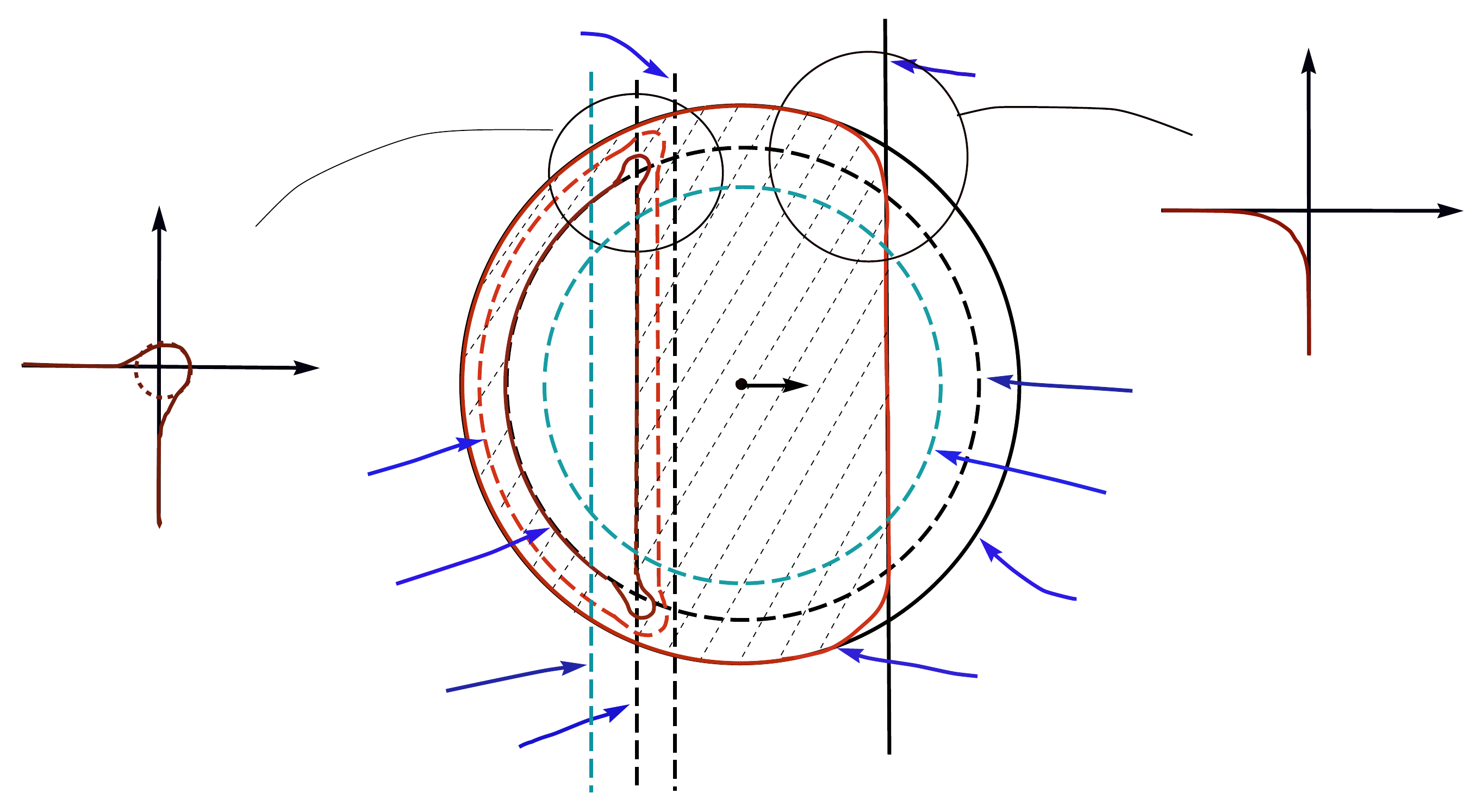}
  \put(78,28){$\mathrm{r}=\|z\|^2=r_2-\delta_1$}
  \put(12,22){$\color{red}  m_{\widetilde{W}_1}=\epsilon_1$}  
  \put(15,15){$\color{red} m_{\widetilde{W}_1}=0$}
  \put(17,7){$\color{cyan} \Re F=-\eta$}
  \put(16,3){$\Re F=\mu-\delta_2$}
  \put(27,52){$\Re F=\mu$}
  \put(66,49){$\Re F=\eta$}
  \put(48,27){$x$}
  \put(54,30){$\xi$}
  \put(74,14){$\mathrm{r}=r_2$}
  \put(76,20){$\color{cyan} \mathrm{r}=r_1$}
  \put(67,8){$\color{red} m_{\widetilde{W}_0}=0$}  
  \end{overpic}
\caption{Construction of $U$: $U$ is the shaded area, and the leftmost and rightmost pictures are illustrating the local smoothing process for the corners.}
\label{figure}
\end{figure}

\begin{lemma}\label{local morse}
$\Gamma_{df}$ is a closed, properly embedded Lagrangian submanifold in
$T^*X$ satisfying
$\overline{\Gamma_{df}}\pitchfork\overline{\Lambda_\mathcal{S}}=\{(x,\xi)\}$
in $\overline{T^*X}$.
\end{lemma}
\begin{proof}
It's a direct consequence of the construction.
\end{proof}

By Proposition \ref{grade_st} and \ref{Pin_st}, $\Gamma_{df}$ can be equipped with a canonical brane structure $b$. Let $L_{x,F}$ denote  $(\Gamma_{df}, \mathcal{E}, b,\Phi)$ as an object in $Fuk(T^*X)$, where $\mathcal{E}$ is a trivial local system of rank 1 on $\Gamma_{df}$ and the perturbation $\Phi=\{L_{x,F}^s\}$ is defined as follows. 

For sufficiently small $s>0$, let 
\begin{equation}\label{U_s}
U_s=\{x\in U: |f(x)|<-\log s\}, U_0=U
\end{equation} and $L^s_{x,F}$ be a Lagrangian over $\overline{U}$ satisfying:  \\
(1) $L^s_{x,F}|_{ \overline{U}_s}=\Gamma_{df}|_{\overline{U}_s}$; \\
(2) $L^s_{x,F}|_{\partial U}=T^*_{\partial U}X|_{|\xi|\geq\beta_s}$, where $\beta_s\rightarrow \infty$ as $s\rightarrow 0$;\\ 
(3) $L^s_{x,F}|_{U}=\Gamma_{df_s}|_{U}$ for a function $f_s$ on $U$;\\ 
(4) $f_s|_{U_s}=f|_{U_s}$, and $df_s|_z=\lambda(z)df|_z$ for some $0<\lambda(z)\leq 1$ for $z\in U-U_s$;\\
(5) let $K^s_{\bar{\eta}}=\min\{|\xi|: \xi\in L_{x,F}^s|_{\partial U_{\bar{\eta}}}\}$, we require $K^s_{\bar{\eta}_1}>K^s_{\bar{\eta}_2}$ for $0<\bar{\eta}_1<\bar{\eta}_2<s$. The notation $L|_{W}$ for a Lagrangian $L$ and a subset $W\subset X$ means the set $L\cap \pi^{-1}(W)$, where $\pi: T^*X\rightarrow X$ is the standard projection.

\subsection{Computation of $Hom_{Fuk(T^*X)}(L_{x,F},
L_{V})$ for $V\in\mathcal{U}_{\mathcal{S}, x}$}\label{df, L_V}

Let $V\in\mathcal{U}_{\mathcal{S}, x}$ (recall the notation is defined in (\ref{U_S_x})), then $\partial V$ is stratified by a subset
$\mathcal{S}_{\partial V}$ of $\mathcal{S}$. Fix a semi-defining function
$m$ for $V$ (see Appendix \ref{def function}). We have the standard Lagrangian brane $L_{V,m}$ associated to
$V$ as defined in Section \ref{basT*X} (b), for which we will simply denote by $L_V$.
Let 
$$V_t=X_{m>t}\text{ for } t\geq 0.$$ 

Let $(\partial U)_{out}$, $(\partial U)_{in}$ denote $\widetilde{W}_0$ and $\widetilde{W}_1$ respectively.
Let $A$ denote the annulus enclosed by $(\partial U)_{in}$, $\mathrm{r}=r_2$ and $\Re(F)=\mu$, including only the boundary component $(\partial U)_{in}$.

\begin{lemma}\label{pairs}
For $t>0$ sufficiently small, there is a compatible collection of quasi-isomorphisms of complexes
\begin{align}
&(\Omega((\overline{U}-(\partial U)_{out})\cap V_t, (\partial U)_{in}\cap V_t),d)\nonumber\\
\simeq& (\Omega((\overline{U}-(\partial U)_{out})\cap V_t, A\cap V_t),d)\label{identity 1}\\
\simeq&(\Omega(B_{r_2}(x)\cap V_t, B_{r_2}(x)\cap \{\Re(F)<\mu\}\cap V_t)
,d)\label{identity 2}\\
\simeq &(\Omega(B_{r_2}(x)\cap V, B_{r_2}(x)\cap \{\Re(F)<\mu\}\cap V),d)\label{identity 3}
\end{align}
\end{lemma}
\begin{proof}
For (\ref{identity 1}), we only need to prove that $A\cap V_t$
deformation retracts onto $(\partial U)_{in}\cap V_t$. First, we can construct
a smooth vector field on an open neighborhood of $(\overline{U}-(\partial U)_{out})\cap\{\mu-\delta_2<\Re(F)<\mu\}$
integrating along which gives a deformation retraction from $A$ onto $U\cap \{\Re(F)\leq\mu-\delta_2\}\cap\{m_{\widetilde{W}_1}\geq 0\}\cup (\partial U)_{in}$. For example, we can choose the vector field $\mathbf{v}$ such that
$$\mathbf{v}(\Re(F))=-1; \mathbf{v}(\mathrm{r})=0 \text{ near }\partial B_{r_2}(x); \mathbf{v}(m)=0\text{ near }\partial V_t; \mathbf{v}(m_{\widetilde{W}_1})\neq 0\text{ on }\widetilde{W}_1.$$
Similarly, we can construct a deformation retraction from $U\cap \{\Re(F)\leq\mu-\delta_2\}\cap\{m_{\widetilde{W}_1}\geq 0\}\cup (\partial U)_{in}$ onto $(\partial U)_{in}\cap V_t$.

The identification in (\ref{identity 2}) is by excision on the triple $(V_t\cap \{m_{\widetilde{W}_1}< 0\})
\subset V_t\cap B_{r_2}(x)\cap \{\Re(F)<\mu\}\subset V_t\cap B_{r_2}(x)$,
 and a deformation retraction from $V_t\cap B_{r_2}(x)$ onto $V_t\cap \{m_{\widetilde{W}_0}< 0\}$. 
One can construct a similar vector field for this and we omit the details. The quasi-isomorphism (\ref{identity 3}) can be also obtained in a similar way.
 \end{proof}

Let $L^t_{V}, t>0$ small, be a family of perturbations of $L_V$ supported over 
$\overline{V_t}$ satisfying:\\
(1) $L^t_{V}|_{\overline{V}_{2t}}=L_V|_{\overline{V}_{2t}}$,\\
(2) $L^t_{V}|_{\partial V_t}=T^*_{\partial V_t}X|_{|\xi|\geq \lambda_t}$, for some $\lambda_t>0$;\\
(3) $L^t_{V}|_{V_t}=\Gamma_{dh_{V_t}}$, where $h_{V_t}$ is a function on $V_t$ such that $h_{V_t}|_{V_{2t}}=\log m|_{V_{2t}}$, $dh_{V_t}$ and $d\log m$ are colinear on $V_t$
 and $1\leq \frac{dh_{V_t}}{d\log m}\leq 1.2$.

Again by by Proposition \ref{grade_st} and \ref{Pin_st}, $L_{V}$ carries a canonical brane structure.
Let $L_V$ also denote the object in $Fuk(T^*X)$ consisting of the canonical brane structure,
a trivial local system of rank 1 on it and the above perturbation $\{L_{V}^t\}_{t\geq 0}$. The proof of the following lemma is essentially the same as in Section 6 of \cite{NZ}. The only difference is that we use the above conical perturbations and avoid geodesic flows.
\begin{lemma}\label{2branes}
There is a fringed set $R\subset \mathbb{R}_+^2$,
such that for $(t,s)\in R$, there is a compatible collection of
quasi-isomorphisms
$$Hom_{Fuk(T^*X)}(L^s_{x,F}, L^t_V)\simeq\Omega (V\cap B_{r_2}(x), V\cap B_{r_2}(x)\cap\{\Re(F)<\mu\}).$$
\end{lemma}
\begin{proof}
Step 1. \emph{Perturbations and dilations.} \\
This step is essentially the same as the perturbation process in $Mor(X)$ stated at the beginning of Section \ref{OpenMor}. Nevertheless, we repeat it to set up notations.
There is an (nonempty) open interval $(0,\eta_0)$ such that for all $t\in (0,\eta_0)$, $\partial V_t\cap \partial U$
transversally. Fixing any $t\in (0,\eta_0)$, there is an open neighborhood $W_t$ of $\partial V_t
\cap \partial U$ such that the covectors $d\log m|_z$ and $L_{x,F}^s|_z$ are linearly independent for all $s>0$ and $z\in W_t\cap \overline{V_t}\cap \overline{U}$. Choose $t<\bar{t}<\eta_0$ such that $X_{t\leq m\leq \bar{t}}\cap \partial U\subset W_t$. Since $X_{t\leq m\leq\bar{t}}- W_t$ is compact, we can find $\epsilon_t>0$ such that $|dh_{V_t}|>\epsilon_t|df|$ on $(X_{t\leq m_{V}\leq\bar{t}}- W_t)\cap\overline{U}$. There is a small $\eta_1>0$ such that $(\overline{U}-U_{\eta_1})\cap X_{t\leq m_V\leq \bar{t}}\subset W_t$ and on $(\overline{U}-U_{\eta_1})\cap \overline{V}_t-W_t$ we have $|dh_{V_t}|$ bounded above by some $M_t$, so we can find $\eta_1>\bar{s}>s>0$ small enough so that $\{\epsilon_t|\xi|: \xi\in L_{x,F}^s|_{\overline{U}-U_{\bar{s}}}\}$ is bounded below by $2M_t$. In summary, we first choose $t, \bar{t}$, then $\epsilon_t$ and lastly $s, \bar{s}$, and it's clear that the collection of such $(t,s)$ forms a fringed set in $\mathbb{R}_+^2$. It is also clear that we can choose $\bar{s}$ small so that $(t,\bar{t})\times(s,\bar{s})\subset R$. 

There is a Riemannian metric $g$ on $X$
such that after a small perturbation, $( h_{V_t}-\epsilon_tf_s, g)$ is a directed pair on the manifold with corners $\overline{V_{\bar{t}}}\cap \overline{U_{\bar{s}}}$, and choices of such metric form
an open convex subset. And this also holds for any $(\tilde{s},\tilde{t})\in (t,\bar{t})\times(s,\bar{s})$.

Step 2. \emph{Energy bound.}\\
Choose $t<t_1<t_2<t_3<\bar{t}$ and $s<s_1<s_2<s_3<\bar{s}$.

Let $G_1=L_V^t|_{X_{t_2<m<t_3 }}$ and $G_2=\epsilon_t\cdot L_{x,F}^s|_{U_{s_2}-\overline{U}_{s_3}}$. Choose some very small $\delta_{i,v}>0$ and define the tube-like open set
$$T_i=\bigcup\limits_{(x,\xi)\in G_i}B^v_{\delta_{i,v}}(x,\xi),$$
where $B^v_{\delta_{i,v}}(x,\xi)$ means the vertical ball in the cotangent fiber $T^*_xX$ of radius $\delta_{i,v}$ centered at $(x,\xi)$.  With small enough $\delta_{i,v}$, we have $T_1\cap T_2=\emptyset$. 

Let $L_{V}^{t,\ell}=\varphi_{D^{\log m}_{t_3,\bar{t}}}^\ell(L_V^t)$ and $L_{x,F}^{s,\ell}=\varphi_{D^f_{s_3,\bar{s}}}^{\epsilon_t\ell}(\epsilon_t\cdot L_{x,F}^s)$ for $0<\ell<1$, where $\varphi_{D^{\log m}_{t_3,\bar{t}}}^\ell$ and $\varphi_{D^f_{s_3,\bar{s}}}^{\epsilon_t\ell}$ are the variable dilations defined in Section \ref{basT*X} (c). Note that the variable dilations fix $G_1, G_2$ and $\overline{L_V^{t,\ell}}\cap \overline{L_{x,F}^{s,\ell}}=(1-\ell)\cdot (\overline{L_V^t}\cap \epsilon_t\cdot\overline{L_{x,F}^s})$.

By Proposition \ref{EB} and Remark \ref{U.E.B}, for $\ell$ sufficiently close to 1,  all the discs bounding $L_V^{t,\ell}$ and $L_{x,F}^{s,\ell}$ have boundaries lying in $L_V^{t,\ell}|_{X_{m>t_2}}\cup L_{x,F}^{s,\ell}|_{U_{s_2}}$. Fixing such an $\ell$, the same holds for the family of uniform dilations $\epsilon\cdot L_V^{t,\ell}$ and $\epsilon\cdot L_{x,F}^{s,\ell}$, $0<\epsilon\leq 1$. 

It is easy to see that $L_{x,F}^{s,\ell}|_{U_{s_1}}$ and $L_V^{t,\ell}|_{V_{t_1}}$ are the graph of differentials of a directed sequence $(U_{s_1}, f_1), (V_{t_1}, f_2)$, and by Fukaya-Oh's theorem (Theorem \ref{FOthm}) and Morse theory for manifolds with corners (\ref{MTMC}), we have for $\epsilon>0$ sufficiently small, 
\begin{align*}
&Hom_{Fuk(T^*X)}(L_{x,F}^s, L_V^t)\simeq Hom_{Fuk(T^*X)}(\epsilon\cdot L_{x,F}^{s,\ell},\epsilon\cdot L_V^{t,\ell})\\
\simeq&Mor^*(\overline{U}_{s_1}\cap\overline{V}_{t_1}, \epsilon(f_2-f_1))\simeq (\Omega((\overline{U}-(\partial U)_{out})\cap V_{t_1}, (\partial U)_{in}\cap V_{t_1}),d)\\
\simeq &(\Omega(B_{r_2}(x)\cap V, B_{r_2}(x)\cap \{\Re(F)<\mu\}\cap V),d).
\end{align*}
The last identity is from Lemma \ref{pairs}.
\end{proof}

\subsection{$H(M_{x,F})\cong H(\mu_X^*Hom_{F(T^*X)}(L_{x,F}, -)$)}\label{simeq}
Given a sequence of open submanifolds with semi-defining functions $(V_i, m_i), i=1,...,k$, there is a fringed set $R\subset\mathbb{R}^{k+1}_{+}$ such that for all $(t_k,...,t_1,t_0)\in R$, there exist $\epsilon_k,...,\epsilon_1,\epsilon_0>0$ and $(t_k,...,t_1,t_0)<(\bar{t}_k,...,\bar{t}_1,\bar{t}_0)\in R$, satisfying \\
(1) $\partial U (\text{resp.} \partial U_{t_0}), \partial V_{t_1},...,\partial V_{t_k}$ intersect transversally, i.e. the unit conormal vectors to them are linearly independent at any intersection point; \\
(2) Let $\Gamma_i^{t_i}=\epsilon_i\cdot L^{t_i}_{V_i}|_{V_{i, \bar{t}_i}}$ for $i=1,...,k$, and $\Gamma_0^{t_0}=\epsilon_0\cdot L^{t_0}_{x,F}|_{U_{\bar{t}_0}}$. Then $\epsilon_i\cdot L^{t_i}_{V_i}\cap \epsilon_j\cdot L^{t_j}_{V_j}=\Gamma_i^{t_i}\cap \Gamma_j^{t_j}$ for $0<i<j$ and $\epsilon_0\cdot L^{t_0}_{x,F}\cap \epsilon_i\cdot L^{t_i}_{V_i}=\Gamma_0^{t_0}\cap \Gamma_i^{t_i}$, for $i>0$;\\
(3) For all $(p_i)_{i=1}^k\in R$ with $(t_i)_{i=0}^k<(p_i)_{i=0}^k<(\bar{t}_i)_{i=0}^k$, $(U_{\bar{t}_0}, \epsilon_0 f_{t_0}), (V_1, \epsilon_1 h_{V_{1,t_1}}), ..., (V_k, \epsilon_k h_{V_{k, t_k}})$ is a directed sequence. \\
The notation $(t_i)_{i=0}^k<(s_i)_{i=0}^k$ means $t_i<s_i$ for all $0\leq i\leq k$. The procedure of choosing the constants is by induction and we put the details at the end of Section \ref{Comp. mod.}.

Again, we do variable
dilations to $\epsilon_0\cdot L_{x,F}^{t_0}, \epsilon_i\cdot L_{V_i}^{t_i}$ and run 
the energy bound argument on holomorphic discs. Choose $(t_i)_{i=0}^k<(p_i)_{i=0}^k<(q_i)_{i=0}^k<(s_i)_{i=0}^k<(\bar{t}_i)_{i=0}^k$ in $R$. Let 
$$\tilde{L}_{x,F}^{t_0,
\ell}=\varphi_{D^f_{s_0,\bar{t}_0}}^{\epsilon_0\ell}(\epsilon_0\cdot L^{t_0}_{x,F}),
\ \tilde{L}_{V_i}^{t_i,\ell}=\varphi_{D^{\log m_i}_{s_i,\bar{t}_i}}^{\epsilon_i\ell}(\epsilon_i\cdot L_{V_i}^{t_i})\text{ for }i>1, $$
and
$$\tilde{L}_{x,F}^{t_0,
\ell}|_{U_{p_0}}=\Gamma_{d\tilde{f}_{t_0,\ell}},\ \tilde{L}_{V_i}^{t_i,\ell}|_{V_{i,p_i}}=\Gamma_{d\tilde{h}_{i,t_i}}$$
for some function $\tilde{f}_{t_0,\ell}$ on $U_{p_0}$, and $\tilde{h}_{i,t_i,\ell}$
on $V_{i,p_i}$ for $1\leq i\leq k$.

For $\ell$ sufficiently close to $1$, all holomorphic discs bounding these Lagrangians have boundaries lying in $\tilde{L}_{x,F}^{t_0,
\ell}|_{U_{q_0}}\cup\bigcup\limits_{i=1}^k \tilde{L}_{V_i}^{t_i,\ell}|_{V_{i,q_i}}$. Since the sequence $(U_{p_0}, \tilde{f}_{t_0,\ell}), (V_{1,p_1}, \tilde{h}_{1,t_1,\ell}),...,(V_{k,p_k}, \tilde{h}_{k,t_k,\ell})$ is directed, using Fukaya-Oh's theorem, we get the following. 

\begin{lemma}\label{tree}
For $\ell$ sufficiently close to 1, 
\begin{align*}
m_{Fuk(T^*X)}^k: &Hom_{Fuk(T^*X)}(\tilde{L}^{t_{k-1},\ell}_{V_{k-1}},
\tilde{L}^{t_k,\ell}_{V_k})\otimes\cdots\otimes Hom_{Fuk(T^*X)}(\tilde{L}_{x,F}^{t_0,\ell},
\tilde{L}_{V_1}^{t_1,\ell})\\
&\rightarrow Hom_{Fuk(T^*X)}(\tilde{L}_{x,F}^{t_0,\ell}, \tilde{L}_{V_k}^{t_k,\ell})[2-k]
\end{align*} is
given by counting Morse trees:
$$m_{Fuk(T^*X)}^k(a_{k-1},\cdots,a_0)=\sum\limits_T\sum\limits_{a_k\in \epsilon\cdot(\Gamma_{d\tilde{f}_{t_0,\ell}}\cap \Gamma_{d\tilde{h}_{k,t_k,\ell}})}\#\mathcal{M}(T; \epsilon \tilde{f}_{t_0,\ell},\epsilon
\tilde{h}_{1,t_1,\ell},\cdots,\epsilon \tilde{h}_{k,t_k,\ell};
\pi(a_0),
\cdots, \pi(a_k))^{0\text{-d}}\cdot a_k,
$$ for all $\epsilon$ sufficiently close to 0, where $\pi: T^*X\rightarrow X$
is the standard projection.
\end{lemma}

Consider the following diagram:
\[
\xymatrix{
\mathcal{B}=F_{\mathcal{S}}(T^*X) \ar@/^/[drr]^{\qquad\mathcal{F}=Hom_{F(T^*X)}(L_{x,F},-)} \\
& \widetilde{\mathcal{B}}=Mor_\mathcal{S}(X) \ar@{_{(}->}[ul]_{i} \ar@/^/[d]^{\mathcal{I}} \ar[r]_{\qquad\mathcal{F}|_{\widetilde{B}}} & Ch \\
& \widetilde{\mathcal{A}}=Open_\mathcal{S}(X) \ar@/^/[u]^{\mathcal{P}} \ar@{^{(}->}[d]^{j} & \\
& \mathcal{A}=Sh_\mathcal{S}(X) \ar@/^/[uuul]^{\mu_X} \ar@/_/[uur]_{M_{x,F}} & }
\] 
Here $i: \widetilde{\mathcal{B}}\hookrightarrow \mathcal{B}$ and $j: \widetilde{\mathcal{A}}\hookrightarrow \mathcal{A}$ are embeddings into triangulated envelopes; the functors $\mathcal{I}, \mathcal{P}$ are from applying
the homological perturbation formalism to the functor $P$ in (\ref{Proj}); $\mu_X: \mathcal{A}\rightarrow \mathcal{B}$ is the microlocal functor in Section \ref{microlocal}.

In Remark \ref{HPL}, putting $G$ and $\tilde{G}$ to be the idempotent $P$ in (\ref{Proj}) on corresponding complexes,  $\mathcal{M}$ to be $M_{x,F}|_{\widetilde{\mathcal{A}}}$ and $\mathcal{F}$ to be $\mathcal{I}$, gives us the $\mathcal{N}$ exactly the same as $\mathcal{F}|_{\widetilde{\mathcal{B}}}$. This follows directly from Lemma \ref{tree}, and we have 
$$H(j^*M_{x,F})\cong
 H(\mathcal{P}^*\mathcal{F}|_{\mathcal{B}})\cong H(j^*\mu_X^*\mathcal{F}).$$
Therefore,
\begin{equation}\label{functor_corresp}
H(\mu_X^*\mathcal{F})\cong H(M_{x,F}).
\end{equation}

\section{Computation of $Hom_{Fuk(T^*X)}(L_{x,F},-)$ on holomorphic branes in $Fuk_\mathcal{S}(T^*X)$}\label{SectionHF}
\subsection{Holomorphic Lagrangian Branes.}

Let $X$ be a compact 
complex manifold of dimension $n$.
Let $T^*X_{\mathbb{C}}$ denote the holomorphic cotangent bundle of $X$ equipped
with the standard holomorphic symplectic form $\omega_{\mathbb{C}}$. 
Like in the real case, there is a complex projectivization of $T^*X_{\mathbb{C}}$,
namely 
$$\overline{T^*X_\mathbb{C}}=(T^*X_\mathbb{C}\times \mathbb{C}-T_X^*X_\mathbb{C}\times\{0\})/\mathbb{C}^*.$$
For a holomorphic (complex analytic) Lagrangian $L$ in $T^*X_\mathbb{C}$ which is by assumption a $\mathcal{C}$-set in $\overline{T^*X}$, using Theorem 4.4 in \cite{PS}, one sees that $\overline{L}$ is complex analytic in $\overline{T^*X}_\mathbb{C}$. Note if $X$ is a proper algebraic variety, then $L$ is algebraic in $T^*X_{\mathbb{C}}$. 

There
is the standard identification (of real vector bundles) $\phi:T^*
X_\mathbb{C}\rightarrow T^*X$ as follows. In local coordinates
$(q_{z_j}, p_{z_j})$ on $T^*X_\mathbb{C}$ and $(q_{x_j}, q_{y_j},
p_{x_j}, p_{y_j})$ on $T^*X$, where $z_j=x_j+\sqrt{-1}y_j$ on $X$, 
we have $q_{x_j}=\Re q_{z_j}, q_{y_j}=\Im q_{z_j},
p_{x_j}=\Re p_{z_j}, p_{y_j}=-\Im p_{z_j}$. It's easy to check that
$\phi^*\omega=\Re \omega_\mathbb{C}$, so $\phi$ sends every holomorphic
Lagrangian to a Lagrangian. In the following, by a holomorphic
Lagrangian in $T^*X$, we mean an \emph{exact} Lagrangian which is the image $\phi(L)$ of a 
holomorphic Lagrangian $L$ in $T^*X_\mathbb{C}$ under the identification $\phi$. 
We will write $L$ instead of $\phi(L)$ when there is no cause of confusion.

Equip $T^*X$ with the Sasaki almost complex structure $J_{Sas}$ and let
$\eta$ be the canonical trivialization of the bicanonical bundle
$\kappa$ (See Appendix \ref{grading_L}).

First we have the following lemma on the flat case $X=\mathbb{C}^n$ (we don't need
$X$ to be compact here), where $\eta$ is the volume form $\Omega=\bigwedge_{i=1}^n (dq_{x^i}+\sqrt{-1}dp_{x^i})\wedge
(dq_{y^i}+\sqrt{-1}dp_{y^i})$ up to a positive scalar. 

\begin{prop}\label{grading}
Every holomorphic Lagrangian brane in $T^*X$ ($X=\mathbb{C}^n$) has an integer grading with respect to $J_{Sas}$.
\end{prop}
\begin{proof}
Let $L$ be a holomorphic Lagrangian in $T^*X_{\mathbb{C}}$. For any
$(x,\xi)\in L$, let $v_1,...,v_k,w_1,...,w_{n-k}$ be a basis of
$T_{(x,\xi)}L$. After a change of coordinate and basis, we can assume that
$v_i=\partial_{q_{z^i}}+\sum\limits_{\mu=1}^n
v_{i}^\mu\partial_{p_{z^\mu}}$ and $w_j=\sum\limits_{\mu=1}^n
w_j^\mu\partial_{p_{z^\mu}}$ for $i=1,...,k$, and $j=1,...,n-k$.

Then the condition of $L$ being a Lagrangian implies that
$w_1,...,w_{n-k}$ generate
$\langle\partial_{p_{z^i}}\rangle_{i=k+1,...,n}$ and after another change
of basis, we could get $v_i=\partial_{q_{z^i}}+\sum\limits_{\mu=1}^k
v_{i}^\mu\partial_{p_{z^\mu}}$ with
$(v_i^\mu)_{i,\mu\in\{1,...,k\}}$ a symmetric $k\times k$
matrix.

Let $J_1=\left(\begin{array}{cc}0&-1\\1&0\end{array}\right)$ and let
$J_{m}=\left(\begin{array}{cccccc}J_1&\ &\ &\ \\ \ &J_1&\ &\ \\
\ &\ &\ddots&\ \\  \ &\ &\ &J_1
\end{array}\right)$ of size $2m\times 2m$. Then $T_{\phi(x,\xi)} \phi(L)$ has the form
$\left(\begin{array}{cc|cc}I_{k}&\mathbf{0}&A&\mathbf{0}\\
\mathbf{0}&\mathbf{0}&\mathbf{0}&K
\end{array}\right)$ (by this, we mean $T_{\phi(x,\xi)}\phi(L)$ is spanned by the row
vectors of the matrix under the basis $\partial_{q_{x^1}},
\partial_{q_{y^1}},...,\partial_{q_{x^n}},
\partial_{q_{y^n}},\partial_{p_{x^1}},
\partial_{p_{y^1}},...,\partial_{p_{x^n}},
\partial_{p_{y^n}}$) where $I_k$ is the identity matrix of size $2k\times 2k$,
$A$ is a symmetric matrix satisfying
$AJ_{k}=-J_{k}A$, and
$K=\left(\begin{array}{ccccc}1&\ &\ &\ &\ \\ \ &-1&\ &\ &\ \\ \ &\
&\ddots &\ &\ \\ \ &\ &\ &1 &\ \\ \ &\ &\ &\ &-1
\end{array}\right)$ of size $2(n-k)\times 2(n-k)$. In particular, $KJ_{n-k}=-J_{n-k}K$.

Let $\Omega=\bigwedge_{i=1}^n (dq_{x^i}+\sqrt{-1}dp_{x^i})\wedge
(dq_{y^i}+\sqrt{-1}dp_{y^i})$ be a holomorphic volume form on
$T^*X$ with respect to $J_{Sas}$. Then for any basis $u_1,...,u_{2n}$ of
$T_{\phi(x,\xi)}\phi(L)$, $(\Omega(u_1\wedge\cdots\wedge
u_{2n}))^2=C\cdot(\det(I_k+\sqrt{-1}A)\det(\sqrt{-1}K))^2$ where $C>0$. Since $AJ_k=-J_kA$, for any eigenvector $v$ of $A$ with eigenvalue $\lambda$, we have $A(J_{k} v)=-\lambda(J_{k} v)$. In particular, if $1+\lambda\sqrt{-1}$ is an eigenvalue of
$I_{k}+\sqrt{-1}A$ then $1-\lambda\sqrt{-1}$ is an
eigenvalue of it as well, and they are of the same multiplicity. So
$(\Omega(u_1\wedge\cdots\wedge u_{2n}))^2$ is always a positive
number, which implies that $L$ has integer grading. 
\end{proof}

In the general case of $X$, for any small disc
$D=\{\sum\limits_i|z_i|^2<\epsilon\}\subset X$, let $J_D$ be the Sasaki almost complex
structure induced from a metric on $X$ which is flat on $D$. Given a
graded holomorphic Lagrangian $L$, deform $J_{con}$ (relative to infinity) to agree with $J_D$ on a relatively compact subset of $T^*X|_{D}$. Lemma \ref{grading} says that it gives a new grading on $L$ which has integer value
on that subset. Since the space of compatible almost complex structures which agree with $J_{con}$ near infinity is
contractible and $X$ is connected, the integer on each connected component of $L$ is 
independent of $D$, and this constant has the same amount of
information as the original grading of $L$.  Because of this, we
will by some abuse of language say that every holomorphic Lagrangian
has integer grading.
\begin{prop}\label{degree}
Let $L_0, L_1$ be two holomorphic Lagrangians in $T^*X$ with integer
gradings $\theta_0, \theta_1$ respectively. Assume that $L_0$ and
$L_1$ intersect transversally. Then $\text{HF}^\bullet(L_0,L_1)$ is
concentrated in degree $\theta_1-\theta_0+n$.
\end{prop}
\begin{proof}
Let $p\in L_0\cap L_1$. By the proof of Proposition  \ref{grading} and
transversality, under
one coordinate system $T_pL_0$ has the form $\left(\begin{array}{cc|cc}I_{k}&\mathbf{0}&A_0&\mathbf{0}\\
\mathbf{0}&\mathbf{0}&\mathbf{0}&K_0
\end{array}\right)$ and $T_pL_1$ has the form $\left(\begin{array}{cc|cc}\mathbf{0}&\mathbf{0}&K_1&\mathbf{0}\\
\mathbf{0}&I_{l}&\mathbf{0}&A_1
\end{array}\right)$ where $k+l\geq n$, $A_i, i=0,1$ is of the same
type as $A$ and $K_i, i=0,1$ is of the same type as $K$ in the proof
of Proposition \ref{grading}.

We find the degree of $p$ using (\ref{def_deg}). First, let 
$$M_0=\left(\begin{array}{cc}
I+\sqrt{-1}A_0&\mathbf{0}\\ \mathbf{0}&\sqrt{-1}K_0
\end{array}\right),\ M_1=\left(\begin{array}{cc}\sqrt{-1}K_1&\mathbf{0}\\
\mathbf{0}& I+\sqrt{-1}A_1\end{array}\right),$$
and 
$$U_i=M_i(\Re M_i^2+\Im M_i^2)^{-\frac{1}{2}}, i=0,1, \ \widetilde{U}=U_0U_1U_0^{-1},$$
$$C=\Re\widetilde{U}\Im\widetilde{U}^{-1}, U=(C+\sqrt{-1}I)(C^2+I)^{-\frac{1}{2}}.$$
It is easy to see that $U_i, U\in U(2n)$, under an orthonormal basis of $T_pL_0$, we have
$$T_pL_1=U\cdot T_pL_0,$$
and the eigenvalues together with eigenvectors of $U$ will give the canonical short path from $T_pL_0$ to $T_pL_1$.

Let $S'=\{B\in
M_{2n\times 2n}(\mathbb{C}): BJ_{n}=-J_{n}\overline{B}\}$. 
Then for a matrix in $S'$, its eigenvalues
are of the form $\lambda_i,-\overline{\lambda}_i,i=1,...,n$. It's straightforward to
check that $U\in S'$ (since $U_i\in S'$), so the
eigenvalues of $U$ are $e^{2\pi\sqrt{-1}\alpha_i},
e^{2\pi\sqrt{-1}(-\frac{1}{2}-\alpha_i)}, i=1,...,n$, for some $\alpha_i\in(-\frac{1}{2},0)$. Therefore  
$$\text{deg}(p)=
\theta_1-\theta_0-2\sum\limits_{i=1}^n(\alpha_i-\frac{1}{2}-\alpha_i))=\theta_1-\theta_0+n.$$
\end{proof}

Let $Lag(T^*X)$ be the set of all Lagrangian submanifolds in $T^*X$.
Let $Lag_\mathcal{S}(T^*X)=\{L\in Lag(T^*X):
L^\infty\subset
\Lambda_\mathcal{S}^\infty\}$.

\subsection{Transversality of $L_{x,F}$ with $t\cdot L$ for $L\in Lag_{\mathcal{S}}(T^*X)$ and $t>0$ sufficiently small}

For any $L\in Lag(T^*X)$, consider
$\mathcal{L}_{t>0}=\{((x,\xi),t): (x,\xi)\in t\cdot L, t>0\}\subset
T^*X\times\mathbb{R}$ and denote each fiber over $t$ as
$\mathcal{L}_t$. Define
$\text{Conic}(L)=\overline{\mathcal{L}_{t>0}}-\mathcal{L}_{t>0}\subset
T^*X\times \{0\}$, and we also view it inside $T^*X$. Similarly, if $X$ is a proper
algebraic variety and 
$L$ is a holomorphic Lagrangian (hence algebraic) in $T^*X_{\mathbb{C}}$, consider
$\mathcal{L}_{w\in\mathbb{C}^*}=\{((x,\xi),w): (x,\xi)\in w\cdot L, w\in\mathbb{C}^*\}
\subset T^*X_{\mathbb{C}}\times \mathbb{C}^*$. Define $\text{Conic}(L^{alg})$ to be the
fiber at $0$ of the algebraic closure of $\mathcal{L}_{w\in\mathbb{C}^*}$ in $T^*X_{\mathbb{C}}\times \mathbb{P}^1$. 

Let $$Cone(L^\infty)=\text{Cl}{\{(x,\xi)\in T^*X: \lim\limits_{s\in \mathbb{R}_+, s\rightarrow \infty} (x,s\xi)\in L^\infty\text{ in }\overline{T^*X}\}}\subset T^*X.$$
For a holomorphic Lagrangian $L$,  
$L^\infty_\mathbb{C}$ will denote $\overline{L}\cap T^\infty X_{\mathbb{C}}
\subset \overline{T^*X}_{\mathbb{C}}$, and let 
$$Cone_{\mathbb{C}^*}(L^\infty_{\mathbb{C}})=\text{Cl}{\{(x,\xi)\in T^*X_{\mathbb{C}}: \lim\limits_{\lambda\in \mathbb{C}^*, \lambda\rightarrow \infty} (x, \lambda\xi)\in 
L^\infty_{\mathbb{C}}\text{ in }\overline{T^*X}_{\mathbb{C}}\}}\subset T^*X_{\mathbb{C}},$$
where Cl means taking closure.
In particular, $Cone(L^\infty)\subset Cone_{\mathbb{C}^*}(L_{\mathbb{C}}^\infty)$, so $L\in Lag_{\mathcal{S}}(T^*X)$ for a complex stratification $\mathcal{S}$. 

In the following, $\mathcal{S}$ is a refinement of a complex stratification, with each stratum a cell, and $L\in Lag_{\mathcal{S}}(T^*X)$.

\begin{lemma}
$\text{Conic}(L)$ is a closed (possibly singular) conical Lagrangian
in $T^*X$. 
\end{lemma}
\begin{proof}
$\text{Conic}(L)=Cone(L^\infty)\cup \overline{\pi(L)} $. In fact, $\text{Conic}(L)\cap
T^*_XX=\overline{\pi(L)}$ and $(x,\xi)\in\text{Conic}(L)-
T^*_XX\Leftrightarrow \exists (x_n,\xi_n)\in L$ and $t_n\rightarrow
0^+$ such that
$\lim\limits_{n\rightarrow\infty}(x_n,t_n\xi_n)=(x,\xi)\Leftrightarrow
\lim\limits_{t\rightarrow\infty}(x,t\xi)=\lim\limits_{n\rightarrow\infty}(x_n,\xi_n)\in
L^\infty$.
\end{proof}

Since $\mathcal{L}_{t>0}=\{((x,\xi),t): (x,\xi)\in t\cdot L,
t>0\}$ is a $\mathcal{C}$-set in $T^*X\times\mathbb{R}$, 
$\text{Conic}(L)$ is a $\mathcal{C}$-set. Note that $\text{Conic}(L)\subset\Lambda_{\mathcal{S}}$, so we can take
a stratification of $\Lambda_{\mathcal{S}}$ which is compatible with $\text{Conic}(L)$.  
Then choose a stratification $\mathcal{T}$ of $\mathcal{L}_{t\geq 0}:=\overline{\mathcal{L}_{t>0}}$ compatible with the above stratification
restricted to $\text{Conic}(L)$. It is clear that for any covector
$(x,\xi)$ in an open stratum in $\Lambda_{\mathcal{S}}$ away from $\text{Conic}(L)$,
$L_{x,F}\cap t\cdot L=\emptyset$ for $t>0$ sufficiently small, for any test triple $(x,\xi,F)$. So we only need to
look at $(x,\xi)$ in an open stratum of $\text{Conic}(L)$.

Let $T_\alpha$ be an open stratum of $\text{Conic}(L)$. For any
$((x,\xi),0)\in T_\alpha$, there is some open neighborhood of
it that only intersects open strata in $\mathcal{L}_{t>0}$, and let $(x,\xi,F)$ be a test triple for $\Lambda_{\mathcal{S}}$.
Denote $\mathcal{L}_{x,F}=L_{x,F}\times \mathbb{R}\subset
T^*X\times\mathbb{R}$.
\begin{lemma}\label{transverse1}
In a neighborhood
of $((x,\xi),0)$, $\mathcal{L}_{x,F}$ intersects $\mathcal{L}_{t>0}$
transversally.
\end{lemma}
\begin{proof}
For a small (open) ball $B_r(x)$ with center $x$, $\pi^{-1}(B_r(x))\subset T^*X$ is diffeomorphic to $D^n\times \mathbb{R}^n$, where $D^n$ is the (open) unit disc in $\mathbb{R}^n$. So we have two $\mathcal{C}$-maps by taking tangent spaces: \\
$$f_1: \mathcal{L}_{t>0}\rightarrow \text{Gr}_{n+1}(\mathbb{R}^{2n+1}), f_2: \mathcal{L}_{x,F}\rightarrow \text{Gr}_{n+1}(\mathbb{R}^{2n+1})$$
and by restriction, these give the following map
$$f=(f_1,f_2): \mathcal{L}_{t>0}\cap \mathcal{L}_{x,F}\rightarrow \text{Gr}_{n+1}(\mathbb{R}^{2n+1})
\times \text{Gr}_{n+1}(\mathbb{R}^{2n+1}).$$ 
Let 
$N=\{(A,B)\in \text{Gr}_{n+1}(\mathbb{R}^{2n+1}) \times
\text{Gr}_{n+1}(\mathbb{R}^{2n+1}): A+B\neq\mathbb{R}^{2n+1}\}$. It is
clear that $N$ is a closed $\mathcal{C}$-set.

Suppose there is a sequence of points $((x_n, \xi_n), t_n)\in T_\beta,
t_n>0$ approaching $((x,\xi), 0)$, where $T_\beta$ is an open stratum in $\mathcal{L}_{t>0}$,
on which $\mathcal{L}_{x,F}$ and $\mathcal{L}_{t>0}$ intersect
nontransversally, then $f((x_n,\xi_n), t_n)\in N$.
Since $N$ is compact, there exists a subsequence
$((x_{n_k}, \xi_{n_k}), t_{n_k})$ such that
 $f((x_{n_k},\xi_{n_k}), t_{n_k})$ converges to a point in $N$ and $\lim\limits_{k\rightarrow\infty}T_{((x_{n_k},
\xi_{n_k}), t_{n_k})}\mathcal{L}_{t>0}$ exists.

By the Whitney property,
$T_{(x,\xi)}\text{Conic}(L)\subset\lim\limits_{k\rightarrow\infty}T_{(x_{n_k},
\xi_{n_k}, t_{n_k})}\mathcal{L}_{t>0}$. This implies that $L$ is not
transverse to $L_{x,F}$ at $(x,\xi)$, which is a contradiction.
\end{proof}

\begin{lemma}\label{transverse2}
Let $((x,\xi),0)\in T_\alpha$ as in Lemma \ref{transverse1}. Then $L_{x,F}$ intersects $t\cdot L$ transversally for all sufficiently small
$t>0$, and the intersections are within the holomorphic portion of $L_{x,F}$.
\end{lemma}
\begin{proof}
First, by curve selection lemma, given any neighborhood $W$ of $((x,\xi),0)$, we have $\mathcal{L}_{x,F}\cap \mathcal{L}_{t_0>t>0}\subset W$ for $t_0$ 
sufficiently small. 

We only need to prove that there is a neighborhood of $((x,\xi),0)$ contained in 
$W$, such that for any $((x_t,\xi_t),t)\in
\mathcal{L}_t\cap\mathcal{L}_{x,F}$, $0<t<t_0$, we have
$\pi_*T_{(x_t,\xi_t,t)}(\mathcal{L}_{t>0}\cap\mathcal{L}_{x,F})\neq \{0\}$
where $\pi:\mathcal{L}_{t>0}\cap\mathcal{L}_{x,F}\rightarrow\mathbb{R}$ is
the projection to $t$. In fact, since 
$\mathcal{L}_{t>0}\pitchfork\mathcal{L}_{x,F}$ in $W$, 
$\pi_*T_{(x_t,\xi_t,t)}(\mathcal{L}_{t>0}\cap\mathcal{L}_{x,F})\neq
\{0\}$ implies the transversality of
$L_t$ and $L_{x,F}$.

The assertion is true because the function $t$ on $\mathcal{L}_{t>0}\cap\mathcal{L}_{x,F}$ has no critical value in
$(0,\eta)$ for some $\eta>0$ sufficiently small.
\end{proof}

Now we are ready to prove the main theorem. 
\begin{thm}\label{main thm}
If $L$ is a holomorphic Lagrangian brane of constant grading $-n$ and $\mathcal{F}\in Sh(X)$ quasi-represents $L$, i.e. $\mu_X(\mathcal{F})\simeq L$, then $\mathcal{F}$ is a perverse sheaf.
\end{thm}

\begin{proof}
From (\ref{micro_S}),
$\mathcal{F}\in Sh_{\mathcal{S}}(X)$ for a complex analytic stratification $\mathcal{S}$. Let $\tilde{\mathcal{S}}$ be a refinement of $\mathcal{S}$ with each stratum a cell. By Proposition \ref{degree}, for generic choices of test triple $(x,\xi, F)$ of 
$\Lambda_{\tilde{\mathcal{S}}}$, we have the cohomology of $M_{x,F}(\mathcal{F})\simeq Hom_{Fuk(T^*X)}(L_{x,F},L)$ concentrate in degree $0$, so $\mathcal{F}$ is perverse.
\end{proof}
\begin{remark}
One could easily deduce from the above discussions that if $\mathcal{F}\in Perv(X)$ quasi-represents a holomorphic
brane $L$, then $Conic(L)=SS(\mathcal{F})$ and in particular $Cone(L^\infty)=\phi(Cone_{\mathbb{C}^*}(L_{\mathbb{C}}^\infty))$.
\end{remark}

Recall the Morse-theoretic definition of the characteristic cycle $CC(\mathcal{F})$
for $\mathcal{F}\in Sh_{\mathcal{S}}(X)$ (see e.g. \cite{SV}; in general $X$ only needs to 
be a real oriented analytic manifold). 
Consider $\bigcup\limits_{S_\alpha\in\mathcal{S}}T^*_{S_\alpha}X-D^*_{S_\alpha}X
=\bigcup\limits_{i\in I}\Lambda_i$, where $\Lambda_i, i\in I$ are disjoint connected 
components.
\begin{definition}
The \emph{characteristic cycle $CC(\mathcal{F})$} of a sheaf $\mathcal{F}\in Sh_{\mathcal{S}}(X)$ is the Lagrangian
cycle with values in the orientation sheaf of $X$ defined as follows:\\
(1) The orientation on $\Lambda_i, i\in I$ are induced from the canonical orientation of
$T^*_{S_\alpha}X$;\\
(2) The multiplicity of $\Lambda_i$ is equal to $\chi(M_{x,F}(\mathcal{F}))$, where 
$(x,dF_x)\in\Lambda_i$ and $(x,dF_x,F)$ is a test triple for $\Lambda_{\mathcal{S}}$. 
\end{definition}

\begin{cor}
If $X$ is a proper algebraic variety, $\mathcal{F}$ and $L$ are as in Theorem \ref{main thm}, and $L$ is equipped with a vector bundle of rank $d$ with flat connection, then $CC(\mathcal{F})=d\cdot Conic(L^{alg})^{sm}$.
\end{cor}
\begin{proof}
If $X$ is a proper algebraic variety, then $Conic(L^{alg})$ is an algebraic cycle whose multiplicity
at a smooth point $(x,\xi)$ is equal to the intersection number of $L_{x,F}$ with $\mathcal{L}_t$ for $t>0$ sufficiently small, and this is equal to the Euler characteristic of the local Morse group 
at $(x,\xi)$ quotient by the rank of the vector bundle.
\end{proof}

\subsection{Some generalization}
Holomorphic branes are very restrictive. First, they have strong conditions on $CC(\mathcal{F})$ for the sheaf $\mathcal{F}$ it represents or equivalently $Conic(L^{alg})$ if $X$ is proper algebraic. For example, on $T^*\mathbb{P}^1$, we cannot have
a connected $L$ with $Conic(L^{alg})$ equal to the sum of the zero section and one cotangent fiber, each of which has multiplicity 1. Second, fixing $Conic(L^{alg})$, they cannot produce all the perverse sheaves with this characteristic cycle. For example, on $T^*\mathbb{P}^1$, let $Conic(L^{alg})=T^*_XX+T^*_{z=0}X+T^*_{z=\infty}X$, then the only candidates for connected $L$ are meromorphic sections of the holomorphic bundle $T^*\mathbb{P}^1\rightarrow \mathbb{P}^1$ which have simple poles at $0$ and $\infty$.  One can show that up to a positive multiple, only $\Gamma_{\frac{dz}{z}}$ and $\Gamma_{-\frac{dz}{z}}$ are exact Lagrangians. So there are only two kinds of perverse sheaves coming in this way: one is $i_*\mathcal{L}_U[1]$ on $\mathbb{P}^1-\{\infty\}$ and $i_!\mathcal{L}_U[1]$ on $\mathbb{P}^1-\{0\}$, where $\mathcal{L}_U$ is a rank 1 local system on $U=\mathbb{P}^1-(\{0\}\cup \{\infty\})$, and the other is its Verdier dual. 

For this reason, we consider a broader class of branes which may produce more perverse sheaves, namely, the branes which are holomorphic near infinity, and are multi-graphs near the zero section. 

\begin{prop}
Let $L$ be a connected Lagrangian brane in $T^*X$. Assume that $L^\infty\neq \emptyset$,  there is $r>0$ such that $L\cap T^*X|_{|\xi|>r}$ is complex analytic on which it has grading $-n$, and $L\cap T^*X|_{|\xi|\leq r+\epsilon}$ is a multi-graph if $\overline{\pi(L)}=T^*_XX$, i.e. $\pi|_{L\cap T^*X|_{|\xi|\leq r}}$ is a submersion. Then $L$ quasi-represents a perverse sheaf $\mathcal{F}$.
\end{prop}
\begin{proof}
We only need to check the local Morse group on the zero section.

If $\overline{\pi(L)}\neq T^*_XX$, then $M_{x,F}(\mathcal{F})\simeq 0$ for $(x,dF_x=0)$ a generic point
on the zero section.

If $\overline{\pi(L)}= T^*_XX$, take a generic point $(x,0)$ on the zero section and construct a local Morse brane $L_{x,F}$. Over a small ball $B_r(x)$ of $x$ in $X$, $\pi^{-1}(B_r(x))\cap  L$ is a finite covering plus some holomorphic portion of $L$. Consider $HF(L_{x,F}, t\cdot L)$ for $t>0$ sufficiently small. Since each sheet in the covering connects to a holomorphic part of grading $-n$ by a path along which there is no critical change of the grading, $HF^\bullet (L_{x,F}, t\cdot L)$ is concentrated in degree 0.
\end{proof}
Although one could not represent every perverse sheaf by a holomorphic Lagrangian brane, it is speculative that locally every indecomposable perverse sheaf can be represented by a holomorphic brane.

\section{Appendices}

\appendix

\section{Analytic-Geometric Categories.}\label{agc}
Analytic-Geometric Categories provide a setting on subsets of manifolds and maps between manifolds, where one can always expect reasonable geometry to happen after standard operations. A typical example is if a $C^1$-function $f:X\rightarrow\mathbb{R}$ is in an analytic-geometric category $\mathcal{C}$ and it is proper,  then its critical values form a discrete set in $\mathbb{R}$. For more general and precise statement, see Lemma \ref{critical}. This tells us that the map $f=x^2\sin(\frac{1}{x}):\mathbb{R}\rightarrow\mathbb{R}$ does not belong to any $\mathcal{C}$, and gives us a sense that certain pathological behavior of arbitrary functions and subsets of manifolds are ruled out by the analytic-geometric setting.

The following is a brief recollection of background results from \cite{Miller}. All manifolds here are assumed to be real analytic, unless otherwise specified.

\subsection{Definition} \label{agcdef}
An \emph{analytic-geometric category} $\mathcal{C}$ assigns every analytic manifold $M$ a collection of subsets in $M$, denoted as $\mathcal{C}(M)$, satisfying the following axioms:\\
(1) $M\in \mathcal{C}(M)$ and $\mathcal{C}(M)$ is a Boolean algebra, namely, it is closed under the standard operations $\cap, \cup, (-)^{c}$ (taking complement);\\
(2) If $A\in \mathcal{C}(M)$, then $A\times\mathbb{R}\in\mathcal{C}(M\times\mathbb{R})$;\\
(3) For any proper analytic map $f: M\rightarrow N$, $f(A)\in \mathcal{C}(N)$ for all $A\in \mathcal{C}(M)$;\\
(4) If $\{U_i\}_{i\in I}$ is an open covering of $M$, then $A\in\mathcal{C}(M)$ if and only if $A\cap U_i\in \mathcal{C}(U_i)$ for all $i\in I$;\\
(5) Any bounded set in $\mathcal{C}(\mathbb{R})$ has finite boundary.

It is easy to construct a category $\mathcal{C}$ from these data. Namely, define objects as all pairs $(A,M)$ with $A\in \mathcal{C}(M)$, and a morphism $f: (A, M)\rightarrow (B, N)$ to be a continuous map $f: A\rightarrow B$, such that the graph $\Gamma_f\subset A\times B$ is lying in $\mathcal{C}(M\times N)$. We will always omit the ambient manifolds, and will call $A$ a $\mathcal{C}$-set and $f:A\rightarrow B$ a $\mathcal{C}$-map.

The smallest analytic-geometric category is the \emph{subanalytic category }$\mathcal{C}_{an}$ consisting of subanalytic subsets and continuous subanalytic maps. It is enough to assume that $\mathcal{C}=\mathcal{C}_{an}$ thoughout the paper, but we work in more generality.

\subsection{Basic facts}\label{AGBR}
Here we list several basic facts on analytic-geometric categories that are used in the main content without proof. 

\subsubsection{Derivatives}\label{tangent}
 Let $A$ be a $(C^1,\mathcal{C})$-submanifold of $M$. If $A\in\mathcal{C}(M)$, then its tangent bundle $TA$ is a $\mathcal{C}$-set of $TM$, and its conormal bundle $T^*_AM$ is a $\mathcal{C}$-set in $T^*M$.

\subsubsection{Curve Selection Lemma.}
\begin{lemma}\label{CSL}
Let $A\in\mathcal{C}(M)$. For any $x\in \overline{A}-A$ and $p\in\mathbb{Z}_{>0}$, there is a $\mathcal{C}$-curve, i.e. a $\mathcal{C}$-map $\rho: [0,1)\rightarrow \overline{A}$, of class $C^p$ with $\rho(0)=x$ and $\rho((0,1))\subset A$. 
\end{lemma}

\subsubsection{Defining functions}\label{def function}

For any closed set $A$ in $M$,  a \emph{defining function} for $A$ is a function $f:M\rightarrow \mathbb{R}$ satisfying $A=\{f=0\}$. 

\begin{prop}
For any closed $\mathcal{C}$-set $A$ and any positive integer $p$, there exists a $(C^p, \mathcal{C})$-defining function for $A$.
\end{prop}

\begin{remark}\label{def fcn}
In the main content, we frequently use the notion of a function $f$ satisfying $\{f>0\}=V$ for a given open $\mathcal{C}$-set $V$, and we will call $f$ a \emph{semi-defining function }of $V$. 
\end{remark}

\subsubsection{Whitney statifications.}

(1) Let $M=\mathbb{R}^N$.  A pair of $C^p$ submanifolds $(X,Y)$ in $M$ ($\text{dim} X=n, \text{dim} Y=m$) is said to satisfy the \emph{Whitney property} if \\
(a) (Whitney property A) For any point $y\in Y$ and any sequence $\{x_k\}_{k\in \mathbb{N}}\subset X$ approaching $y$, if $\lim\limits_{k\rightarrow \infty}T_{x_k}X$ exists and equal to $\tau$ in $Gr_n(\mathbb{R}^N)$, then $T_yY\subset \tau$;\\
(b) (Whitney property B) In addition to the assumptions in (a), let $\{y_k\}_{k\in\mathbb{N}}\subset Y$ be any sequence approaching $y$. If the limit of the secant lines $\lim\limits_{k\rightarrow\infty}\overline{x_ky_k}$ exists and equal to $\ell$, then $\ell\subset\tau$.

It is easy to see that Whitney property B implies Whitney property A. 
The Whitney property obviously extends for any manifold $M$, just by covering $M$ with local charts.

(2) A $C^p$ \emph{stratification} of a closed subset $P$ is a locally finite partition by $C^p$-submanifolds $\mathcal{S}=\{S_\alpha\}_{\alpha\in\Lambda}$ satisfying $$S_\alpha\cap\overline{S}_\beta\neq \emptyset, \alpha\neq \beta\Rightarrow S_\alpha\subset\overline{S}_\beta-S_\beta.$$

A \emph{Whitney stratification} of $P$ in class $C^p$ is a $C^p$ stratification $\mathcal{S}=\{S_\alpha\}_{\alpha\in\Lambda}$ such that every pair $(S_\alpha, S_\beta)$ 
satisfies the Whitney property. 

We will also need the following notions:

(i) We say that a collection of subsets in $M$, $\mathcal{A}$, is \emph{compatible} with another collection of subsets $\mathcal{B}$, if for any $A\in \mathcal{A}$ and $B\in\mathcal{B}$, we have either $A\cap B=\emptyset$ or $A\subset B$.

(ii) Two stratifications $\mathcal{S}$ and $\mathcal{T}$ are said to be \emph{transverse}, if for any $S_\alpha\in\mathcal{S}$ and $T_\beta\in\mathcal{T}$, we have $S_\alpha\pitchfork T_\beta$. It is clear that 
$$\mathcal{S}\cap\mathcal{T}:=\{S_\alpha\cap T_\beta: S_\alpha\in\mathcal{S}, T_\beta\in\mathcal{T}\}$$
is also a stratification.

(iii) Let $f: P\rightarrow N$ be a $C^1$-map, and $\mathcal{S}, \mathcal{T}$ be $C^p$-stratifications of $P$ and $N$ respectively. The pair $(\mathcal{S},\mathcal{T})$ is called a $C^p$-\emph{stratification of }$f$ if $f(S_\alpha)\in\mathcal{T}$ for all $S_\alpha\in\mathcal{S}$ and $df|_{S_\alpha}$ is a submersion. 

Now assume $\mathcal{C}=\mathcal{C}_{\text{an}}$, $\mathcal{C}_{\text{an}}^{\mathbb{R}}$ or $\mathcal{C}_{\text{an}, \exp}$ (see the definitions in \cite{Miller}).

\begin{prop}\label{Whitney}
Let $P$ be a closed $\mathcal{C}$-set in $M$. Let $\mathcal{A},\mathcal{B}$ be collections of $\mathcal{C}$-sets in $M, N$ respectively. \\
(a) There is a $C^p$-Whitney stratification $\mathcal{S}\subset \mathcal{C}(M)$ of $P$ that is compatible with $\mathcal{A}$, and has connected and relatively compact strata. \\
(b) Let $f: P\rightarrow N$ be a proper $(C^1,\mathcal{C})$-map. Then there exists a $C^p$-Whitney stratification $(\mathcal{S}, \mathcal{T})\subset \mathcal{C}(M)\times\mathcal{C}(N)$ of $f$ such that $\mathcal{S}$ and $\mathcal{T}$ are compatible with $\mathcal{A}$ and $\mathcal{B}$ respectively, and have connected and relatively compact strata.

One can make the strata in (a), (b) to be all cells. 
\end{prop}

(iv) For any $C^p$ Whitney stratification $\mathcal{S}$ of $M$, define its associated cornomal 
$$\Lambda_\mathcal{S}:=\bigcup\limits_{S_\alpha\in\mathcal{S}}T^*_{S_\alpha}X.$$
Let $f: X\rightarrow\mathbb{R}$ be a $C^1$-map. We say $x\in X$ is a $\Lambda_\mathcal{S}$-\emph{critical point of }$f$ if $df_x\in\Lambda_\mathcal{S}$. We say $w\in\mathbb{R}$ is a $\Lambda_\mathcal{S}$-\emph{critical value of }$f$ if $f^{-1}(w)$ contains a $\Lambda_\mathcal{S}$-critical point. More generally, let $\mathbf{f}=(f_1,...,f_n): M\rightarrow \mathbb{R}^n$ be a proper $(C^1, \mathcal{C})$-map. We say that $x$ is a \emph{critical point} of $\mathbf{f}$ if there is a nontrivial linear combination of $(df_i)_x, i=1,...,n$ contained in $\Lambda_{\mathcal{S}}$. Similarly, $\mathbf{w}\in\mathbb{R}^n$ is called a \emph{critical value of} $\mathbf{f}$ if $\mathbf{f}^{-1}(\mathbf{w})$ contains a critical point. Otherwise, $\mathbf{w}$ is called a \emph{regular value of} $\mathbf{f}$.

If in addition $\mathcal{S}\subset\mathcal{C}(M)$ and $f: X\rightarrow\mathbb{R}$ is a proper $\mathcal{C}$-map, then we apply Curve Selection Lemma (Lemma \ref{CSL}) and have
\begin{lemma}\label{critical}
The $\Lambda_\mathcal{S}$-critical values of $f$ form a discrete subset of $\mathbb{R}$. 
\end{lemma}

We will need the following variant of the notion of a fringed set from \cite{Goresky}, which is also used in \cite{NZ}.
\begin{definition}\label{fringed}
A \emph{fringed set} $R$ in $\mathbb{R}_+^n$ is an open subset satisfying the following properties.  
For $n=1$, $R=(0,r)$ for some $r>0$.
For  $n>1$, the image of $R$ under the projection $\mathbb{R}_+^n\rightarrow \mathbb{R}_+^{n-1}$ to the first $n-1$ entries is a fringed set in $\mathbb{R}_+^{n-1}$, and 
if $(r_1,...,r_{n-1},r_n)\in R$, then $(r_1,...,r_{n-1},r'_n)\in R$ for all $r'_n\in (0, r_n)$.
\end{definition}

\begin{cor}\label{fringed_cor}
Let $\mathbf{f}=(f_1,...,f_n): M\rightarrow \mathbb{R}^n$ be a proper $(C^1, \mathcal{C})$-map. Then there is a fringed set $R\subset \mathbb{R}^n_+$ consisting of $\Lambda_\mathcal{S}$-regular values of $\mathbf{f}$.
\end{cor}

\subsection{Assumptions on $X$ and Lagrangian submanifolds in $T^*X$}
\label{assumption on Lag}
Throughout the paper, $X$ is assumed to be a compact real analytic
manifold or compact complex manifold. Then $T^*X$ is real analytic.
The projectivization $\overline{T^*X}=(T^*X\times\mathbb{R}_{\geq
0}-T^*_XX\times\{0\})/\mathbb{R}^+$ is a semianalytic subset in
the manifold $\mathbb{P}_+(T^*X\times\mathbb{R})=(T^*X\times\mathbb{R}-T^*_XX\times\{0\})/\mathbb{R}^+$.

Fix an analytic-geometric category $\mathcal{C}$. Define
$\mathcal{C}$-sets in $\overline{T^*X}$ to be $\mathcal{C}$-sets in
$\mathbb{P}_+(T^*X\times \mathbb{R})$ intersecting
$\overline{T^*X}$. All Lagrangian submanifolds $L$ in $T^*X$ are
assumed to satisfy $\overline{L}\subset \overline{T^*X}$ a
$\mathcal{C}$-set in $\overline{T^*X}$. All subsets of $X$ are
assumed to be $\mathcal{C}$-sets unless otherwise specified.

\section{$A_\infty$-categories}\label{A_infty-categories}

Roughly speaking, $A_\infty$-cateogory is a form of category whose structure is
more complicated but more flexible than classical notion of category: composition of morphisms are
not strictly associative but only associative up to higher homotopies, and there are also 
successive homotopies between homotopies. In this section, we will briefly recall the definition of $A_\infty$-category, left and right $A_\infty$-modules and $A_\infty$-triangulation. The materials are from Chapter 1 \cite{Seidel}.

\subsection{$A_\infty$-categories and $A_\infty$-functors}
A \emph{non-unital $A_\infty$-category} $\mathcal{A}$ consists of the following data:\\
(1) a collection of objects $X\in Ob\mathcal{A}$,\\
(2) for each pair of objects $X,Y$, a morphism space $Hom_\mathcal{A}(X,Y)$ which is
a cochain complex of vector spaces over $\mathbb{C}$,\\
(3) for each $d\geq 1$ and sequence of objects $X_0, ..., X_d$, a linear morphism
$$m_{\mathcal{A}}^d: Hom_{\mathcal{A}}(X_{d-1}, X_d)\otimes\cdots\otimes Hom_{\mathcal{A}}(X_0,X_1)
\rightarrow Hom_{\mathcal{A}}(X_0, X_d)[2-d]$$
satisfying the following identities
\begin{equation}\label{A_infty_eq}
\sum\limits_{\substack{k+l=d+1, k,l\geq 1\\ 0\leq i\leq d-l}}(-1)^{\dagger_{i}} m_{\mathcal{A}}^k(a_{d},...,a_{i+l+1},m_{\mathcal{A}}^l(a_{i+l},...,a_{i+1}), a_{i},...,a_1)=0,
\end{equation}
where $\dagger_{i}=|a_1|+\cdots+|a_{i}|-i$ and $a_j\in Hom_\mathcal{A}(X_{j-1}, X_{j})$ for 
$1\leq j\leq d$.

A special case of an $A_\infty$-category is a dg-category where all the higher compositions
$m_{\mathcal{A}}^d, d\geq 3$ vanish.

From the above definition, at the cohomological level for $[a_1]\in H(Hom_{\mathcal{A}}(X_0, X_1),m^1_{\mathcal{A}})$ and $[a_2]\in H(Hom_{\mathcal{A}}(X_1, X_2), m^1_{\mathcal{A}})$, their composition
$[a_2]\cdot [a_1]:=(-1)^{|a_1|}[m^2_{\mathcal{A}}(a_2,a_1)]\in H(Hom_{\mathcal{A}}(X_0, X_2)), m_{\mathcal{A}}^1)$ is well defined, and it is easy to check that the product is associative. We will let $H(\mathcal{A})$ denote the non-unital graded category arising in this way. There is also a subcategory
$H^0(\mathcal{A})\subset H(\mathcal{A})$ which only has morphisms in degree 0.

An $A_\infty$-category is called \emph{$c$-unital} if $H(\mathcal{A})$ is unital. All the $A_\infty$-categories we encounter throughout this paper are $c$-unital unless otherwise specified. We will always omit the prefix $c$-unital at those places. One major benefit of dealing with $c$-unital $A_\infty$-categories is that one can talk about quasi-equivalence between categories, see below. 

Given two non-unital $A_\infty$-categories $\mathcal{A}$ and $\mathcal{B}$, 
a \emph{non-unital $A_\infty$-functor }$\mathcal{F}:\mathcal{A}\rightarrow \mathcal{B}$
assigns each $X\in Ob\mathcal{A}$ an object $\mathcal{F}(X)$ in $\mathcal{B}$, and it consists for every $d\geq 1$ and sequence of objects $X_0, ..., X_d\in Ob\mathcal{A}$, a linear morphism 
$$\mathcal{F}^d: Hom_{\mathcal{A}}(X_{d-1}, X_d)\otimes\cdots\otimes Hom_{\mathcal{A}}(X_0,X_1)
\rightarrow Hom_{\mathcal{B}}(\mathcal{F}(X_0), \mathcal{F}(X_d))[1-d],$$
satisfying the identities
\begin{align*}
&\sum\limits_{k\geq 1}\sum\limits_{\substack{s_1+...+s_k=d,\\ s_i\geq 1}}m_{\mathcal{B}}^k(\mathcal{F}^{s_k}(a_d,...,a_{d-s_k+1}),...,
\mathcal{F}^{s_1}(a_{s_1},...,a_1))\\
=&\sum\limits_{\substack{k+l=d+1, k,l\geq 1\\ 0\leq i\leq d-l}}(-1)^{\dagger_{i}}\mathcal{F}^k(a_d,...,a_{i+l+1},m_{\mathcal{A}}^l(a_{i+l},...,a_{i+1}), a_{i},...,a_1)).
\end{align*}
The composition of two $A_\infty$-functors $\mathcal{F}: \mathcal{A}\rightarrow \mathcal{B}$ and $\mathcal{G}:\mathcal{B}\rightarrow \mathcal{C}$ is defined as
$$(\mathcal{G}\circ \mathcal{F})^d(a_d,...,a_1)=\sum\limits_{k\geq 1}\sum\limits_{\substack{s_1+...+s_k=d,\\ s_i\geq 1}}\mathcal{G}^k(\mathcal{F}^{s_k}(a_d,...,a_{d-s_k+1}),...,
\mathcal{F}^{s_1}(a_{s_1},...,a_1)).$$

 It is clear that $\mathcal{F}$ descends on the cohomological level to a functor from $H(\mathcal{A})$ to $H(\mathcal{B})$, which we will denote by $H(\mathcal{F})$.
One easy example of a functor from $\mathcal{A}$ to itself is the identity functor $id_\mathcal{A}$, which is identity on objects and hom spaces and $id_{\mathcal{A}}^k=0$ 
for $k\geq 2$.

Let $\mathcal{Q}=nu\text{-}fun(\mathcal{A},\mathcal{B})$ be the $A_\infty$-category of non-unital $A_\infty$-functors
from $\mathcal{A}$ to $\mathcal{B}$ defined as follows. An element $T=(T^0, T^1,...)$ of degree $|T|=g$, called a \emph{pre-module homomorphism}, in $hom_{\mathcal{Q}}(\mathcal{F}, \mathcal{G})$ is a sequence of linear maps
$$T^d: Hom_{\mathcal{A}}(X_{d-1}, X_d)\otimes\cdots\otimes Hom_{\mathcal{A}}(X_0,X_1)
\rightarrow Hom_{\mathcal{B}}(\mathcal{F}(X_0),\mathcal{G}(X_d))[g-d],$$
in particular, $T^0$ is an element in $Hom_{\mathcal{B}}(\mathcal{F}(X),\mathcal{G}(X))$ of degree $g$ for each $X$.

We also have the following structures
\begin{align*}
&(m^1_{\mathcal{Q}}(T))^d(a_d,...,a_1)\\
&=\sum\limits_{1\leq i\leq k}\sum\limits_{\substack{s_1+\cdots+s_k=d,\\ s_i\geq 0, s_j\geq 1, j\neq i}}(-1)^\dagger m^k_{\mathcal{B}}(\mathcal{G}^{s_k}(a_d,...,a_{d-s_k+1}),...,\mathcal{G}^{s_{i+1}}(a_{s_1+\cdots+s_{i+1}},...,a_{s_1+\cdots+s_{i}+1}),\\
&T^{s_i}(a_{s_1+\cdots+s_{i}},...,a_{s_1+\cdots+s_{i-1}+1}),\mathcal{F}^{s_{i-1}}(a_{s_1+\cdots+s_{i-1}},...,a_{s_1+\cdots+s_{i-2}+1}),...,
\mathcal{F}^{s_1}(a_{s_1},\cdots,a_1))\\
&-\sum\limits_{\substack{r+l=d+1, r,l\geq 1\\ 1\leq i\leq d-l}}(-1)^{\dagger_{i}+|T|-1}\sum T^r(a_d,...,a_{i+l+1},m_{\mathcal{A}}^l(a_{i+l},...,a_{i+1}), a_{i},...,a_1))
\end{align*}

If we write the right hand side of the above formula for short as 
$$\sum m_{\mathcal{B}}(\mathcal{G},...,\mathcal{G}, T,\mathcal{F},...,\mathcal{F})-\sum T(id,...,id,m_{\mathcal{A}}, id,...,id),$$ then for 
$T_0\in Hom_{\mathcal{Q}}(\mathcal{F}_0,\mathcal{F}_1)$ and $T_1\in Hom_{\mathcal{Q}}(\mathcal{F}_1,\mathcal{F}_2)$, we have
$$m^2_{\mathcal{Q}}(T_1, T_0)=\sum m_{\mathcal{B}}(\mathcal{F}_2,...,\mathcal{F}_2, T_1,\mathcal{F}_1,...,\mathcal{F}_1,T_0,\mathcal{F}_0,...,\mathcal{F}_0),$$ and similar
formulas apply to higher differentials $m_{\mathcal{Q}}^d$ for $d>2$. Note that 
there is no $m_{\mathcal{A}}$ involved in $m_{\mathcal{Q}}^d$ for $d\geq 2$.

Those $T$ for which $m^1_{\mathcal{Q}}(T)=0$ are the \emph{module homomorphisms}, and $H(T)$ in $H(\mathcal{Q})$ descends to a natural transformation between $H(\mathcal{F})$ and $H(\mathcal{G})$ under the map
$$H(nu\text{-}fun(\mathcal{A},\mathcal{B}))\rightarrow Nu\text{-}fun(H(\mathcal{A}),H(\mathcal{B})).$$ 

Assume $\mathcal{F},\mathcal{G}:\mathcal{A}\rightarrow\mathcal{B}$ are two $A_\infty$-functors such that $\mathcal{F}(X)=\mathcal{G}(X)$ for every $X\in \text{Ob}(\mathcal{A})$. Then $\mathcal{F}$ and $\mathcal{G}$ is called \emph{homotopic} if there is $T\in hom^{-1}_{\mathcal{Q}}(\mathcal{F},\mathcal{G})$ such $m^1_{\mathcal{Q}}(T)^d=\mathcal{G}^d-\mathcal{F}^d$. We have $H(\mathcal{F})=H(\mathcal{G})$ if $\mathcal{F}$ and $\mathcal{G}$ are homotopic.

Let $\mathcal{A}$, $\mathcal{B}$ be $c$-unital $A_\infty$-categories, a functor $\mathcal{F}: \mathcal{A}\rightarrow \mathcal{B}$ is called $c$-unital if $H(\mathcal{F})$ is unital. Then the full subcategory $fun(\mathcal{A},\mathcal{B})\subset\mathcal{Q}$ consisting of $c$-unital functors is a $c$-unital $A_\infty$-category. 

A $c$-unital functor $\mathcal{F}:\mathcal{A}\rightarrow \mathcal{B}$ is a \emph{quasi-equivalence} if $H(\mathcal{F}): H(\mathcal{A})\rightarrow H(\mathcal{B})$ is an equivalence of categories.

\subsection{$A_\infty$-modules and Yoneda embedding}
In this subsection, we will assume all $A_\infty$-categories to be $c$-unital.

Define the $A_\infty$-category of left $\mathcal{A}$-modules as
$$l\text{-}mod(\mathcal{A})=fun(\mathcal{A},Ch).$$
Explicitly, any $\mathcal{M}\in l\text{-}mod(\mathcal{A})$ assigns a cochain complex
$\mathcal{M}(X)$ to each object $X$ and we have
$$m_{\mathcal{M}}^d: 
Hom_{\mathcal{A}}(X_{d-1}, X_d)\cdots\otimes Hom_{\mathcal{A}}(X_0,X_1)\otimes\mathcal{M}(X_0)\rightarrow \mathcal{M}(X_d)[2-d]$$ 
with the property that
$$\sum m_{\mathcal{M}}(id,..,id, m_{\mathcal{M}})+\sum m_{\mathcal{M}}(id,...,id,m_{\mathcal{A}}, id,...,id)=0,$$
where there is at least one $id$ after $m_{\mathcal{A}}$ in the second term.

An important example of a left $\mathcal{A}$-module is $\tilde{\mathcal{Y}}_{X_0}$ for $X_0\in Ob\mathcal{A}$ defined
as $\tilde{\mathcal{Y}}_{X_0}(X)=Hom_{\mathcal{A}}(X_0,X)$ and $\tilde{\mathcal{Y}}^d_{X_0}$ coincides
with $m^d_{\mathcal{A}}$.

The category of right $\mathcal{A}$-modules $mod(\mathcal{A})$ (following usual convention, we don't denote it by $r\text{-}mod(\mathcal{A})$) can be defined similarly as $fun(\mathcal{A}^{opp}, Ch)$. An important example is $\mathcal{Y}_{X_0}$ defined as
$\mathcal{Y}_{X_0}(X)=Hom_{\mathcal{A}}(X, X_0)$ and this gives the \emph{Yoneda embedding}
\begin{align*}
\mathcal{Y}:&\mathcal{A}\rightarrow mod(\mathcal{A})\\
&X\mapsto \mathcal{Y}_X.
\end{align*}
For $c_i\in Hom_{\mathcal{A}}(Y_{i-1}, Y_i)$, $1\leq i\leq d$,
$$\mathcal{Y}(c_d,...,c_1)^k: \mathcal{Y}_{Y_0}(X_k)\otimes Hom_{\mathcal{A}}(X_{k-1}, X_k)
\otimes\cdots\otimes Hom_{\mathcal{A}}(X_0,X_1)\rightarrow \mathcal{Y}_{Y_d}(X_0)$$
is $m_{\mathcal{A}}^{k+d+1}(c_d,...,c_1,b,a_k,...,a_1)$ for $b\in \mathcal{Y}_{Y_0}(X_k) $ and $a_i\in Hom_{\mathcal{A}}(X_{i-1}, X_i)$.

Note that $mod(\mathcal{A})$ is a dg-category, and the Yoneda embedding $\mathcal{Y}$
is cohomologically full and faithful. This gives a construction showing that every $A_\infty$-category is quasi-equivalent to a (strictly unital) dg-category, i.e. its image under $\mathcal{Y}$.

For $\mathcal{F}: \mathcal{A}\rightarrow \mathcal{B}$, we can define the associated pull-back
functor
\begin{align*}
\mathcal{F}^*:mod(\mathcal{B}) (\text{resp. }l\text{-}mod(\mathcal{B}))&\rightarrow mod(\mathcal{A}) (\text{resp. }l\text{-}mod(\mathcal{A}))\\
\mathcal{M}&\mapsto \mathcal{M}\circ\mathcal{F}.
\end{align*}

\subsection{$A_\infty$-triangulation}
Recall that a triangulated envelope of an $A_\infty$-category $\mathcal{A}$ 
is a pair $(\mathcal{B},\mathcal{F})$ of a triangluated $A_\infty$-category and 
a quasi-embedding $\mathcal{F}: \mathcal{A}\rightarrow \mathcal{B}$ such that
$\mathcal{B}$ is generated by the image of objects in $\mathcal{A}$. We refer the reader to Section 3, Chapter 1 in \cite{Seidel} for the definition of triangulated $A_\infty$-categories. Any two
triangulated envelopes of $\mathcal{A}$ are quasi-equivalent.

There are basically two ways of constructing $A_\infty$-triangulated envelope.
One is to take the usual triangulated closure of the image of $\mathcal{A}$ under the
Yoneda embedding in $mod(\mathcal{A})$, since $mod(\mathcal{A})$ is triangulated.
The other is by taking twisted complexes of $\mathcal{A}$ which we denote by
$Tw(\mathcal{A})$. The formulation in the definition is a little bit long and messy, which we don't 
really need in this paper, so we refer the reader to consult Seidel \cite{Seidel} Section 3 for a
detailed description.

\section{Infinitesimal Fukaya Categories}\label{Fuk}

In this section, we review the definition of infinitesimal Fukaya category on a Liouville manifold, originated from \cite{NZ}. This section is by no means a complete or rigorous exposition of Fukaya categories. One could consult \cite{Auroux} for a comprehensive introduction, and Seidel's book \cite{Seidel} for a complete and rigorous treatment. 

Our goal here is to give a rough idea of how Fukaya category  (in the exact setting) is defined, and what kind of extra structures one should put on the ambient symplectic manifold and on the Lagrangian submanifolds so to give a coherent definition of the $A_\infty$-structure. We also include several specific facts about $Fuk(T^*X)$, which will supplement the main content.

\subsection{Assumptions on the ambient symplectic manifold}\label{Liouville}
Let $(M, \omega=d\theta)$ be a $2n$-dimensional Liouville manifold. By definition, $M$ is obtained by gluing a compact symplectic manifold with contact boundary $(M_0, \omega_0=d\theta_0)$ with an infinite cone $(\partial M_0\times [1,\infty), d(r\theta_0|_{\partial M_0}) )$ along $\partial M_0$, where $r$ is the coordinate on $[1,\infty)$. We require that the the \emph{Liouville vector field} $Z$, defined by the property $\iota_Z\omega=\theta$, is pointing outward along $\partial M_0$, and the gluing is by identifying $Z$ with $r\partial_r$.

Let $J$ be a $\omega$-compatible almost complex structure on $(M,\omega)$,  whose restriction to the cone $\partial M_0\times [S,\infty)$ for $S>>0$, satisfies that $J\partial_r=R$, where $R$ is the Reeb vector field of $r\theta_0|_{\partial M_0\times \{r\}}$, and $J$ preserves $\text{ker}(r\theta_0|_{\partial M_0\times \{r\}})$, on which it is induced from $J|_{\partial M_0\times\{S\}}$. We will call such a $J$ as a \emph{conical} almost complex structure. It is a basic fact that the space of all such almost complex structures is contractible. The compactible metric $g$ will be conical near infinity, i.e. $g=r^{-1}dr^2+S^{-1}rds^2$, where $ds^2=\omega(\cdot, J\cdot)|_{\partial M_0\times \{S\}}$. Let $\mathcal{H}$ be the set of Hamiltonian functions whose restriction to $\partial M_0\times [S,\infty)$ is $r$ for $S>>0$. Note that the Hamiltonian vector field $X_H$ of $H\in\mathcal{H}$ near infinity is $-rR$.

One can compactify $M$ using the cone structure, i.e. $\overline{M}=M_0\cup\{[t_0x: t_1]| x\in\partial M_0, t_0, t_1\in\mathbb{R}^+, t_0^2+t_1^2\neq 0\}$, here $[t_0x: t_1]$ denotes the equivalence class of the relation $(t_0x, t_1)\sim (\lambda t_0 x, \lambda t_1)$ for $\lambda>0$. It is easy to see that $\overline{M}=M\cup M^\infty$, where we think of an element $(x,r)$ in the cone as $[rx:1]$ and the points in $M^\infty$ are of the form $[x:0]$.

\subsection{Floer theory with $\mathbb{Z}/2\mathbb{Z}$-coefficients and gradings}\label{FLTH}
 To obtain well defined Floer theory for noncompact Lagrangian submanifolds, we should be more careful about their behavior near infinity. First we restrict ourselves in some fixed analytic-geometric setting $\mathcal{C}$, and require that the Lagrangians $L$ we are considering satisfy $\overline{L}$ is a $\mathcal{C}$-set in $\overline{M}$ (see \ref{assumption on Lag}). Second, we need to ensure compactness of holomorphic discs with Lagrangian boundary conditions. A sufficient condition for this is the \emph{tameness} condition following \cite{Sikorav}. We will discuss this in more detail in the next section. 

Recall the Floer theory defines for each pair of Lagrangians $L_1, L_2$ in $M$ a $\mathbb{Z}/2\mathbb{Z}$-graded cochain complex 
\begin{align}
\label{CF}&CF^*(L_0, L_1):=(\bigoplus\limits_{p\in L_0\cap L_1}\mathbb{Z}/2\mathbb{Z} \langle p \rangle, \partial_{CF})\\
&\partial_{CF}(p)=\sum\limits_{q\in L_1\cap L_2}\sharp\mathcal{M}(p,q; L_0, L_1)^{0\text{-}d}\cdot q,
\end{align}
where $\mathcal{M}(p,q; L_0, L_1)^{k\text{-d}}$ is the quotient (by $\mathbb{R}$-symmetry) of the $(k+1)$-dimensional locus of the moduli space $\hat{\mathcal{M}}(p,q; L_0, L_1)$ of holomorphic strips, starting from $q$, ending at $p$ and bounding $L_0, L_1$, i.e. a map
\begin{align}
\nonumber&u: \mathbb{R}\times [0,1]\rightarrow M, \text{ such that }\\
\label{limit cond}&\lim\limits_{s\rightarrow -\infty}u(s,t)=q, \lim\limits_{s\rightarrow +\infty}u(s,t)=p\\
\label{bdry cond}&u(\mathbb{R}\times\{0\})\subset L_0, u(\mathbb{R}\times\{1\})\subset L_1\\
\label{d_bar}&(du)^{0,1}=0 (\Leftrightarrow \frac{\partial u}{\partial s}+J(u)\frac{\partial u}{\partial t}=0).
\end{align}
There are always several technical issues to be clarified in the above definition.

(a) \emph{Transverse intersections.} Implicit in (\ref{CF}) is the step of Hamiltonian perturbation to make $L_0$ and $L_1$ transverse. Let $L_i^\infty$ denote $\overline{L}_i\cap M^\infty$.  If $L_0^\infty\cap L_1^\infty=\emptyset$, then one chooses a generic Hamiltonian function $\widetilde{H}$, whose Hamiltonian vector field vanishes on $L_1$ outside a compact region, and replaces $L_1$ by $\phi_{\widetilde{H}}^t(L_1)$ for small $t>0$. If $L_0^\infty\cap L_1^\infty\neq\emptyset$, then one replaces $L_1$ by $\phi_H^t(L_1)$ for a generic $H\in\mathcal{H}$. It can be shown that $\phi_H^t(L_1^\infty)$ will be apart from $L_0^\infty$ for sufficiently small $t>0$. The invariance of Floer theory under Hamiltonian perturbations ensures that the complex $CF^*(L_0, L_1)$ is well defined up to quasi-isomorphisms.

(b) \emph{Regularity of Moduli space of strips.} One views the $\overline{\partial}$-operator on $u$, i.e. $(du)^{0,1}$, as a section of a natural Banach vector bundle over a suitable space of maps $u$ satisfying (\ref{limit cond}) and (\ref{bdry cond}). Then  $\hat{\mathcal{M}}(p,q; L_0, L_1)$ becomes the intersection of $\overline{\partial}$ with the zero section. We need the intersection to be transverse, and this is equivalent to the linearized operator $D_u$ (a Fredholm operator) of $\overline{\partial}$ at any $u\in \overline{\partial}^{-1}(0)$ being surjective. In many good settings (including the cases in $Fuk(T^*X)$ ), this is true for a generic choice of $J$, which we will refer as a \emph{regular} (compatible) almost complex structure. Then by Gromov's compactness theorem, $\mathcal{M}(p,q; L_0, L_1)^{0\text{-d}}$ is a compact manifold, so $\sharp\mathcal{M}(p,q; L_0, L_1)^{0\text{-d}}$ is finite. Different choices of regular $J$'s give cobordant moduli spaces, therefore the number doesn't depend on such choices (note that we are working over $\mathbb{Z}/2\mathbb{Z}$, so we don't need any orientation on $\mathcal{M}(p,q; L_0, L_1)$ to conclude this). More generally, one would need to introduce time-dependent almost complex structures and Hamiltonian perturbations to achieve transversality. 

(c) \emph{$\partial^2_{CF}=0$.} This is ensured when no sphere or disc bubbling occurs, and it holds for a pair of \emph{exact} Lagrangians $L_0, L_1$, i.e. $\theta|_{L_j}$ is an exact 1-form for $j=0,1$. To verify this, one studies the boundary of the 1-dimensional moduli space $\mathcal{M}(p,q; L_0, L_1)^{1\text{-d}}$ of holomorphic strips starting at $q$ and ending at $p$, and realizes that they are broken trajectories corresponding exactly to the terms involving $q$ in $\partial^2_{CF}(p)$. Since the number of boundary points is even, $\partial^2_{CF}=0$.
 
(d) \emph{Gradings.} For any holomorphic strip $u$ connecting $q$ to $p$, the Fredholm index of the linearized Cauchy-Riemann operator $D_u$ ``in principle" gives the relative grading between $p$ and $q$. The index can be calculated by the \emph{Maslov index} of $u$ defined as follows. A strip $\mathbb{R}\times [0,1]$ is conformally identified with the closed unit disc $D$, with two  punctures on the boundary. Then one can trivialize the symplectic vector bundle $u^*TM$ over the closed unit disc, and think of $T_pL_j, T_qL_j$ for $j=0,1$ as elements in the Lagrangian Grassmannian $LGr(\mathbb{R}^{2n}, \omega_0)$, where $\omega_0$ is the standard symplectic form on $\mathbb{R}^{2n}$. 

By a standard fact from linear symplectic geometry, there is a unique set of numbers $\{\alpha_k\in(-\frac{1}{2},0)\}_{k=1,...,n}$ such that relative to an orthonormal basis $\{v_1,...,v_n\}$ of $T_pL_0$, $T_pL_1$ is spanned by $e^{2\pi\sqrt{-1}\alpha_k}v_k$ for $k=1,...,n$. One could consult Lemma 3.3 in \cite{Alston} for a proof.  Since we will use it in Proposition \ref{degree}, we discuss this in a little more detail. First, this property is invariant under $U(n)$-transformation, so we can assume $T_pL_0=\mathbb{R}^n\subset \mathbb{R}^n\oplus \sqrt{-1}\mathbb{R}^n$. There is a standard way to produce a unitary matrix $U$ such that $T_pL_1=U\cdot T_pL_0$, namely choose a symmetric matrix $A$ in $GL_n(\mathbb{R})$ for which $T_pL_1=(A+\sqrt{-1} I)\cdot T_pL_0$, then let $U=(A+\sqrt{-1}I)(A^2+I^2)^{-\frac{1}{2}}$. Also for any $B+\sqrt{-1}C\in U(n)$ statisfying 
$T_pL_1=U\cdot T_pL_0$, we have $B+\sqrt{-1}C=(A+\sqrt{-1}I)(A^2+I^2)^{-\frac{1}{2}} O$ for some $O\in O(n)$, and $BC^{-1}=A$. Now let $\{v_1,...,v_n\}$ be an orthonormal collection of eigenvectors of $A$, hence of $U$ as well, and $e^{2\pi\sqrt{-1}\alpha_1},...,e^{2\pi\sqrt{-1}\alpha_k}$, $\alpha_j\in(-\frac{1}{2},0)$ be their corresponding eigenvalues of $U$. Then $\{\alpha_j\}_{j=1,...,n}$ is the desired collection of numbers. 

Then $\lambda_p(t):=\text{Span}\{e^{2\pi\sqrt{-1}\alpha_j t}v_j\}\in LGr(\mathbb{R}^{2n}, \omega_0), t\in[0,1]$ is the so called \emph{canonical short path} from $T_pL_0$ to $T_pL_1$. Let $\lambda_q$ be the canonical short path from $T_qL_0$ to $T_qL_1$, and $\ell_j, j=0,1$ denote the path of tangent spaces to $L_j$ from $q$ to $p$ in $u^*TM|_{\partial D}$. Then the \emph{Maslov index} of $u$, denoted as $\mu(u)$, is defined to be the Maslov number of the loop by concatenating the paths $\ell_0, \lambda_p, -\ell_1, -\lambda_q$.

In general, $\mu(u)$ depends on the homotopy class of $u$, so wouldn't give well defined relative degree between $p$ and $q$. But if $L_0, L_1$ are both oriented, we have a well defined grading, namely, $\deg(p)=0$ if $\lambda_p$ takes the orientation of $L_0$ into the orientation of $L_1$, otherwise, $\deg{p}=0$. In the next section, we will see that under certain assumptions, we will not only get $\mathbb{Z}/2\mathbb{Z}$-gradings on the Floer complex, but $\mathbb{Z}$-gradings.

(e) \emph{Product structure.} Consider three Lagrangians $L_0, L_1, L_2$, then one can define a linear map
\begin{align*}
&m: CF^*(L_1, L_2)\otimes CF^*(L_0, L_1)\rightarrow CF^*(L_0, L_2)\\
&m(a_1, a_0)=\sum\limits_{a_2\in L_0\cap L_2} \sharp\mathcal{M}(a_0, a_1,a_2; L_0, L_1, L_2)^{0\text{-d}}\cdot a_2.
\end{align*}
$\mathcal{M}(a_0, a_1,a_2; L_0, L_1, L_2)^{0\text{-d}}$ is the 0-dimensional locus of the moduli space of equivalence class of holomorphic maps 
$$u: (D, \{0,1,2\})\rightarrow (M, \{a_0,a_1,a_2\}), u(\overline{i(i+1)})\subset L_i, i\in\mathbb{Z}/(3\mathbb{Z})$$
where $0,1,2$ are three (counterclockwise) marked points on $\partial D$, and $\overline{i(i+1)}$ denotes the arc in $\partial D$ connecting $i$ and $i+1$. The equivalence relation is composition with conformal maps of the domain. Since the conformal structure of a disc with three marked points on the boundary is unique (and there is no nontrivial  automorphism), we can just fix a conformal structure once for all.  

As before one needs to separate $L_0, L_1, L_2$ near infinity if necessary, and the separation process obeys a principle called \emph{propagating forward in time}. Namely one replaces $L_i$ by $\phi_{H_i}^{t_i}(L_i)$, for some $H_i\in\mathcal{H}, i=0,1,2$, and the choices of $(t_2,t_1,t_0)\in\mathbb{R}^3_+$ should be in a fringed set (see Definition \ref{fringed}). The regularity issue about $\mathcal{M}(a_0, a_1,a_2; L_0, L_1, L_2)$ is similar to that of (b).

Similarly to (c), by looking at the boundary of $\mathcal{M}(a_0, a_1,a_2; L_0, L_1, L_2)^{1\text{-d}}$, one concludes the following equation
$$m(\partial_{CF}\cdot, \cdot)+m(\cdot, \partial_{CF}\cdot)+\partial_{CF} m(\cdot, \cdot)=0.$$
This means that $m$ induces a multiplication on the cohomological level $HF^*$. We will see later that $m$ is not strictly associative, but associative up to homotopy.

\subsection{(infinitesimal) Fukaya cateogry of $M$}
The preliminary version of Fukaya category (with $\mathbb{Z}/2\mathbb{Z}$-grading, and over $\mathbb{Z}/2\mathbb{Z}$-coefficients), is an upgrade of the Floer theory, which uncovers much richer structure, the $A_\infty$-structure, of Lagrangian intersection theory. One not only studies $\partial_{CF}$ and $m$, but also studies for each sequence of $n+1$ Lagrangians the \emph{higher compositions }$\mu^n$
\begin{align*}
&\mu^d: CF^*(L_{d-1}, L_{d})\otimes\cdots CF^*(L_1, L_2)\otimes CF^*(L_0, L_1)\rightarrow CF^*(L_0, L_d)[2-d]\\
&\mu^d(a_{d-1},\cdots, a_1, a_0)=\sum\limits_{a_d\in L_0\cap L_d} \sharp\mathcal{M}(a_0, a_1,\cdots, a_d; L_0, L_1, \cdots, L_d)^{0\text{-d}}\cdot a_d,
\end{align*}
where the moduli space $\mathcal{M}(a_0, a_1,\cdots, a_d; L_0, L_1, \cdots, L_d)^{0\text{-d}}$ is defined similarly as before. Assuming regularity of the moduli spaces and no bubblings (ensured by Lagrangians being exact), the boundary of $\mathcal{M}(a_0, a_1,\cdots, a_d; L_0, L_1, \cdots, L_d)^{1\text{-d}}$ gives us the identity \ref{A_infty_eq}. 

Now let's discuss the (final) version of Fukaya category with $\mathbb{Z}$-gradings, and  $\mathbb{C}$-coeffients. We first collect several basic notions about $T^*X$ which we will use in later discussions.

\subsubsection{Some basic notions about $T^*X$}\label{basT*X}

(a) \emph{Almost complex structures.}
Given any Riemannian metric on $X$, there is a $\omega$-compatible almost complex structure,  called \emph{Sasaki almost complex structure}, $J_{Sas}$ on $T^*X$ defined as follows. For any point $(x,\xi)\in T^*X$, there is a canonical splitting 
$$T_{(x,\xi)}T^*X=T_b\oplus T_f,$$
using the dual Levi-Civita connection on $T^*X$, where $T_f$ denotes the fiber direction and $T_b$ denotes the horizontal base direction. The metric also gives an identification $j: T_b\rightarrow T_f$ and it induces a unique almost complex structure, $J_{Sas}$, by requiring $J_{Sas}(v)=-j(v)$ for $v\in T_b$.

Since $T^*X$ is a Liouville manifold, one can use the construction in Section \ref{Liouville} to get a conical almost complex structure, by requiring $J|_{|\xi|=r}=J_{Sas}|_{|\xi|=r}$ for $r$ sufficiently large, and $J=J_{Sas}$ near the zero section. We will denote any of these almost complex structures by $J_{con}$.

(b) \emph{Standard Lagrangians.}
Given a smooth submanifold $Y\subset X$ and a defining function $f$ for $\partial Y$ which is positive on $Y$, we define the \emph{standard Lagrangian}
\begin{equation}\label{L_{Y,f}}
L_{Y,f}=T_Y^*X+\Gamma_{d\log f}\subset T^*X|_Y.
\end{equation} 
It is easy to check that $L_{Y,f}$ is determined by $f|_Y$. 

In the main content, we often restrict ourselves to standard Lagrangians defined by an open submanifold $V$ and a semi-defining function of $V$ (see Remark \ref{def fcn}).

(c) \emph{Variable dilations.}
Consider the class of Lagrangians of the form $L=\Gamma_{df}$, where $f$ is a function on an open submanifold $U$ with smooth boundary, and $\partial U$ decomposes into two components $(\partial U)_{in}$ and $(\partial U)_{out}$ such that $\lim\limits_{x\rightarrow \partial U_{in}} f(x)=-\infty$ and $\lim\limits_{x\rightarrow \partial U_{out}} f(x)=+\infty$.

The \emph{variable dilation} is defined by the following Hamiltonian flow. Choose $0<A<B<1$ and a bump function $b_{A,B}: \mathbb{R}\rightarrow\mathbb{R}$, such that $b_{A,B}(s)=s$ on $[\log B,-\log B]$ and $|b_{A,B}(s)|=-\log\sqrt{AB}$ outside $[\log A,-\log A]$. We assume that $b_{A,B}$ is odd and nondecreasing. Take a function $D^f_{A,B}$ which extends $b_{A,B}\circ \pi^*f$ to the whole $T^*X$. The Hamiltonian flow $\varphi_{D^f_{A,B}}^t$ fixes $L|_{X_{|f|>-\log A}}$, dilates $L|_{X_{|f|<-\log B}}$ by the factor $1-t$, and sends $L$ to a new graph.

\subsubsection{Compactness of moduli space of holomorphic discs: tame condition and perturbations}\label{Comp. mod.}

As we mentioned in the last section,  we need certain tameness condition to ensure the compactess of the moduli space of holomorphic discs bounding a sequence of Lagrangians. The tameness condition adopted here is from Definition 4.1.1 and 4.7.1 in \cite{Sikorav}. $(M, J)$ is certainly a tame almost complex manifold in that sense. For a smooth submanifold $N$, let $d_N(\cdot, \cdot)$ denote the distance function of the metric on $N$ induced from $M$. The tameness requirement on a Lagrangian submanifold $L$ is the existence of two positive numbers $\delta_L, C_L$, such that within any $\delta_L$-ball in $M$ centered at a point $x\in L$, we have $d_L(x, y)\leq C_L d_M(x, y), y\in L$, and the portion of $L$ in that ball is contractible. 

The main consequence of these is the monotonicity property on holomorphic discs from  Proposition 4.7.2 (iii) in \cite{Sikorav}.

\begin{prop}\label{EB}
There exist two positive constants $R_L, a_L$, such that for all $\mathfrak{r}<R_L, x\in M$, and any compact $J$-holomorphic curve $u: (C,\partial C)\rightarrow (B_\mathfrak{r}(x), \partial B_\mathfrak{r}(x)\cup L)$ with $x\in u(C)$, we have $\text{Area}(u)\geq a_L\mathfrak{r}^2$. 
\end{prop}

\begin{remark}\label{U.E.B}
As indicated in \cite{NZ}, the argument of this proposition is entirely local, one could replace the pair $(M, L)$ by an open submanifold $U\subset M$ together with a properly embedded Lagrangian submanifold $W$ in $U$ satisfying the tame condition. In particular, if $M=T^*X$, and $W$ is the graph of differential of a function $f$ over an open set which is $C^1$-\emph{close to the zero section}, i.e. the norm of the partial derivatives of $f$ has uniform bound, then one can dilate $W$ towards the zero section, and get a uniform bound for the family $(\epsilon\cdot U, \epsilon\cdot W)$. More precisely, one could find $R_{\epsilon\cdot W}=\epsilon R_W$ and $a_{\epsilon\cdot W}=a_W$.
\end{remark}

With the monotonicity property, one can show the compactness of moduli of discs bounding a sequence of exact Lagrangians $L_1,...,L_k$ using standard argument. Moreover, assume $M=T^*X$, and consider the class of Lagrangians in Section \ref{basT*X} (c), then we have better control of where holomorphic discs can go bounding a sequence of such Lagrangians, see the proof of Lemma \ref{2branes} and Section 6.5 in \cite{NZ} for more details. 

In the rest of this section, we supplement the details of choosing appropriate $(t_k,...,t_0)$, $\epsilon_k,...,\epsilon_0$ and $(\bar{t}_k,...,\bar{t}_0)$ at the beginning of Section \ref{simeq}.

We start by choosing appropriate $t_k$ and $\epsilon_k$ and then do induction. First, let $$\Lambda_{\neq k}=\bigcup\limits_{i<k}\Lambda_i.$$ 
There is $\eta_k>0$ such that on $(0,\eta_k)$, $m_{k}$ has no $\Lambda_{\neq k}$-critical value. Pick any small $t_k\in (0,\eta_k)$, form $L_{V_k}^{t_k}$. Let $\epsilon_k=1$.

Suppose we have chosen $t_k,...,t_{i+1}$ and
$\epsilon_k,...,\epsilon_{i+1}$ for $i>0$, let $\Lambda_{j,t_j}$ be the
associated conical Lagrangian of the stratification compatible with 
$\{X_{m_j=t_j}\}$, for $j=i+1,...,k$. Let 
$$\Lambda_{\neq
i}=(\bigcup\limits_{j<i}\Lambda_{j})\bigcup
(\bigcup\limits_{j>i}\Lambda_{j,t_j}).$$
There is $\eta_i>0$ so
that $m_i$ has no $\Lambda_{\neq i}$-critical value in
$(0,\eta_i)$.  On $V_i\cap X_{m_j\geq t_j}$ for each $j>i$, there is
an open neighborhood $W_{ij}$ of $X_{m_i=0}\cap X_{m_j=t_j}$ on which $d\log m_i$ and $d\log m_j$ are everywhere linearly independent. Choose $t_j<\eta'_{ji}<\eta_j$ such that $X_{t_j\leq
m_j\leq \eta'_{ji}}\cap \partial V_i\subset W_{ij}$ for $j>i$. On $X_{t_j\leq
m_j\leq \eta'_{ji}}\cap V_i-W_{ij}$, we have the covectors in
$\epsilon_j\cdot L_{V_j}^{t_j}$ bounded from below by some
$N_{ij}>0$. Choose $\epsilon_i>0$ such that on this region, the
covectors in $\epsilon_i\cdot L_{V_i}$ are bounded above by
$\frac{1}{2}N_{ij}$ for all $j>i$. Next, choose
$0<\eta'_{ij}<\eta_i$ so that $X_{m_i\leq \eta'_{ij}}\cap X_{m_j\leq
\eta'_{ji}}\subset W_{ij}$ for all $j>i$. Then covectors in
$\epsilon_j\cdot L_{V_j}^{t_j}$ over $X_{m_i\leq \eta'_{ij}}\cap
X_{m_j\geq t_j}-W_{ij}$ are bounded above by some $M_{ij}$. Choose
$0<t_i<\eta'_{ij}$ for all $j>i$ so that the covectors on the graph
$\epsilon_i\cdot d\log m_i$ over $X_{m_i=t_i}\cap X_{m_j\geq t_j}$
are bounded below by $2M_{ij}$. Now we have $t_i, \epsilon_i$ and
$L_{V_i}^{t_i}$. 

Finally, having chosen $t_k,...,t_1$, and $\epsilon_k,...,\epsilon_1$, to choose $\epsilon_0$, we do the same thing as before. However, to choose $t_0$, we don't shrink $U$. Instead, we find $t_0$ small enough so that on $U-U_{t_0}\cap X_{m_j\geq t_j}$, $\epsilon_0\cdot L_{x,F}^{t_0}$  is bounded below by $2M_{0j}$ for all $j>0$. Clearly, the choices of $(t_k,...,t_0)$ form a fringed set $R\subset \mathbb{R}_+^{k+1}$.

The choices of $\bar{t}_k,...,\bar{t}_0$ can be made as follows. First, $\bar{t}_k$ can be anything satisfying $t_k<\bar{t}_k<\eta'_{ki}$ for all $i<k$. Once we have chosen $\bar{t}_k,...,\bar{t}_{i+1}$ for $i>0$, $\bar{t}_i$ should satisfy $t_i<\bar{t}_i<\eta'_{ij}$ for all $j\neq i$ and on $X_{t_i\leq m_i\leq \bar{t}_i}$, the covectors $d\log m_i$ are bounded below by $1.5M_{ij}$ for all $j>i$. Similar choice can be made for $\bar{t}_0$. Also we can make $(\bar{t}_k,...,\bar{t}_0)$ belong to $R$.

\subsubsection{Gradings on Lagrangians and $\mathbb{Z}$-grading on $CF^*$}\label{grading_L}

Let $\mathcal{L}Gr(TM)=\bigcup\limits_{x\in M}LGr(T_xM, \omega_x)$ be the Lagrangian Grassmannian bundle over $M$. To obtain gradings on Lagrangian vector spaces in $TM$, we need a \emph{universal} Lagrangian Grassmannian bundle $\widetilde{\mathcal{L}Gr(TM)}$, and this amounts to the condition that $2c_1(TM)=0$.  Choose a trivialization $\alpha$ of the bicanonical bundle $\kappa^{\otimes 2}$, and a grading to $\gamma\in LGr(T_xM, \omega_x)$ is a lifting of the \emph{phase map} $\phi(\gamma)=\frac{\alpha(\Lambda^n\gamma)}{|\alpha(\Lambda^n\gamma)|}\in S^1$ to $\mathbb{R}$. 

The condition $2c_1(TM)=0$ holds if $M=T^*X$ for an $n$-dimensional compact manifold $X$. Because the pull back of $\Lambda^n TT^*X$ to the zero section $X$ is just $\mathfrak{o}\mathfrak{r}_X\otimes \mathbb{C}$, where $\mathfrak{o}\mathfrak{r}_X$ is the orientation sheaf on $X$. Since $\mathfrak{o}\mathfrak{r}_X^{\otimes 2}$ is always trivial, and $X$ is a deformation retract of $T^*X$, we get $c_1(TT^*X)$ is 2-torsion. In fact, given a Riemannian metric on $X$, $\mathfrak{o}\mathfrak{r}_X^{\otimes 2}$ is canonically trivialized, and the same for $\kappa^{\otimes 2}$.

For a Lagrangian submanifold $L$ in $M$, we define a \emph{grading} of $L$ to be a continuous lifting $L\rightarrow\mathbb{R}$ to the phase map $\phi_L: L\rightarrow S^1$. The obstruction to this is the \emph{Maslov class} $\mu_L=\phi_L^*\beta\in H^1(L,\mathbb{Z})$, where $\beta$ is the class representing the $1\in H^1(S^1,\mathbb{Z})$. 

\begin{prop}\label{grade_st}
Standard Lagrangians and the local Morse brane $L_{x,F}$ in $T^*X$ both admit canonical  gradings. 
\end{prop}

The reason that all these Lagrangians admit canonical grading is that they are all constructed by (properly embedded) partial graphs over smooth submanifolds. Suppose there is a loop $\Omega\subset L$ such that $(\phi_L)|_{\Omega}: \Omega\rightarrow S^1$ is homotopically nontrivial. Since $\Omega$ is contained in a compact subset of $L$, one can dilate $L$ so that when $\epsilon\rightarrow 0$, $T(\epsilon\cdot L)|_\Omega$ is uniformly close to the tangent planes to the zero section if $L=L_{x,F}$ or to $T_Y^*X$ if $L=L_{Y,f}$. It is easy to check that $T_Y^*X$ has constant phase 1 (resp. $-1$) if $Y$ has even (resp. odd) codimension, so admit canonical grading $0$ (resp. $1$). Then we get a contradiction, because the homotopy type of the map $(\phi_{\epsilon\cdot L})|_{\epsilon\cdot \Omega}: \epsilon\cdot \Omega\rightarrow S^1$ is unchanged under dilation, and $L$ has a canonical grading.

Given two graded Lagrangians $L_i, \theta_i: L_i\rightarrow\mathbb{R}$, $i=0,1$, then for any $p\in L_0\cap L_1$ (assuming transverse intersection), we can define an \emph{absolute $\mathbb{Z}$-grading }of $p$:
\begin{equation}\label{def_deg}
\deg{p}=\theta_1-\theta_0-\sum\limits_{i=1}^n \alpha_i,
\end{equation}
where $\alpha_i, i=1,...,n$ are constants defining the canonical short path from $L_0$ to $L_1$ in Section \ref{FLTH} (d). 

It is easy to check that $\text{ind} (u)=\deg q-\deg p$  for any holomorphic strip $u$ connecting $q$ to $p$ for $q,p\in L_0\cap L_1$, and the absolute $\mathbb{Z}$-grading gives the $\mathbb{Z}$-grading of $CF^*(L_0, L_1)$ for two graded Lagrangians. For more details, see  Section 4 in \cite{Alston}.

\subsubsection{$Pin$-structures.} Recall that $Pin^+(n)$ is a double cover of $O(n)$ with center $\mathbb{Z}/2\mathbb{Z}\times\mathbb{Z}/2\mathbb{Z}$. A \emph{$Pin$-structure} on a manifold $M$ of dimension $n$, is a lifting of the classifying map $M\rightarrow BO(n)$ of $TM$ to a map $M\rightarrow BPin^+(n)$. The obstruction to the existence of a $Pin$-structure is the second Stiefel-Whitney class $w_2\in H^2(M, \mathbb{Z}/2\mathbb{Z})$. The choices of $Pin$-structures form a torsor over $H^1(M, \mathbb{Z}/2\mathbb{Z})$. 

For any class $[w]\in H^2(M, \mathbb{Z}/2\mathbb{Z})$, one could define the notion of a $[w]$-twisted $Pin$-structure on $M$. Fix a  \v{C}ech representative $w$ of $[w]$, and a \v{C}ech cocycle $\tau\in \text{\v{C}}^1(X, O(n))$ representing the principal $O(n)$-bundle associated to $TM$. Then choose a \v{C}ech cochain $\tilde{w}\in \text{\v{C}}^1(X, Pin^+(n))$ which is a lifting of $\tau$ under the exact sequence
$$0\rightarrow \text{\v{C}}^1(X, \mathbb{Z}/2\mathbb{Z})\rightarrow \text{\v{C}}^1(X, Pin^+(n))\rightarrow \text{\v{C}}^1(X, O(n))\rightarrow 0.$$
We say $\tilde{w}$ defines a $[w]$-twisted $Pin$-structure if the \v{C}ech-coboundary of $\tilde{w}$, which obviously lies in the subset $\text{\v{C}}^2(X, \mathbb{Z}/2\mathbb{Z})$, is equal to $w$. It is clear that the definition doesn't essentially depend on the choice of cocycle representatives, and the set of $[w]$-twisted $Pin$-structures, if nonempty, forms a torsor over
$H^1(M, \mathbb{Z}/2\mathbb{Z})$. 

Fix background class $[w]\in H^2(M, \mathbb{Z}/2\mathbb{Z})$, for any submanifold $L\subset M$, define a \emph{relative $Pin$-structure} on $L$ to be a $[w]|_L$-twisted $Pin$-structure. Here we fix a \v{C}ech-representive of $[w]$, and use it for all $L$. Note that the existence of a relative $Pin$-structure only depends on the homotopy class of the inclusion $L\hookrightarrow M$. 

Now let $M=T^*X$ and fix $\pi^*w_2(X)$ as the background class in $H^2(M,\mathbb{Z}/2\mathbb{Z})$ and a relative $Pin$-structure on the zero section. For any smooth submanifold $Y\subset X$, the metric on $X$ gives a canonical way (up to homotopy) to identify $T^*_YX$ near the zero section with a tubular neighborhood of $Y$ in $X$, hence there is a canonical relative $Pin$-structure on $T^*_YX$ by pulling back the fixed relative $Pin$-structure on $X$.  Since the inclusion $L_{Y,f}\hookrightarrow M$ in (\ref{L_{Y,f}}) is canonically homotopic to the inclusion $T_Y^*X\hookrightarrow M$ by dilation, and similarly for $L_{x,F}\hookrightarrow M$ with $T_U^*X\hookrightarrow M$, we have the following
 \begin{prop}\label{Pin_st}
The Lagrangians $L_{Y,f}$ and $L_{x,F}$ have canonical $Pin$-structures.
\end{prop}

\subsubsection{Final definition of $Fuk(M)$}
Fix a background class in $H^2(M,\mathbb{Z}/2\mathbb{Z})$.

\begin{definition}
A \emph{brane structure} $b$ on a Lagrangian submanifold $L\subset M$ is a pair $(\tilde{\alpha}, P)$, where $\tilde{\alpha}$ is a grading on $L$ and $P$ is a relative $Pin$-structure on $L$.
\end{definition}

Recall that we need tame Lagrangians to ensure compactness of moduli of discs, but there are many Lagriangians, e.g. many standard Lagrangians in $T^*X$, which are not tame, but admit appropriate perturbations by tame Lagrangians. Therefore the following is introduced in \cite{NZ}.

\begin{definition}
A \emph{tame perturbation} of $L$ is a smooth family of tame Lagrangians $L_t, t\in\mathbb{R}$, with $L_0=L$ such that\\
(1) Ristricted to the cone $\partial M_0\times [1,\infty)$, the map $t\times r: L_t|_{r>S}\rightarrow\mathbb{R}\times (S,\infty)$ is a submersion for $S>>0$;\\
(2) Fix a defining function $m_{\overline{L}}$ for $\overline{L}\subset\overline{M}$, we require that for any $\epsilon>0$, there exists $t_\epsilon>0$ such that $L_t\subset N_\epsilon(L):=\{m_{\overline{L}}<\epsilon\}$ for $|t|<t_\epsilon$.
\end{definition}
Note it is enough to define the family over an open interval of $0$ in $\mathbb{R}$. 

Now we define $Fuk(M)$. An object in $Fuk(M)$ is a triple $(L,b,\mathcal{E})$ together with a tame perturbation $\{L_t\}_{t\in\mathbb{R}}$ of $L$, where $(L,b)$ is an exact Lagrangian brane, $\mathcal{E}$ is a vector bundle with flat connection on $L$. It is clear that any element in the perturbation famility $L_t$ canonically inherits a brane structure, and a vector bundle with flat connection from $L$. In the following, we still use $L$ to denote an object.

It is proved in Lemma 5.4.5 of \cite{NZ} that every standard Lagrangian admits a tame perturbation.  So for each pair $(U,m)$ of an open submanifold $U\subset X$ and a semi-defining function $m$ of $U$, there is a \emph{standard object} in $Fuk(T^*X)$, which is the standard Lagrangian $L_{U,m}$ equipped with the canonical brane structures, a trivial rank 1 local system and the perturbation in Lemma 5.4.5 of \cite{NZ}.

The morphism space between $L_0$ and $L_1$ is the $\mathbb{Z}$-graded Floer complex enriched by the vector bundles
$$Hom_{Fuk(M)}(L_0, L_1)=: \bigoplus\limits_{p\in L_0\cap L_1}Hom(\mathcal{E}_0|_p, \mathcal{E}_1|_p)\otimes_{\mathbb{C}}\mathbb{C}\langle p \rangle[-\deg (p)].$$
Implicit in the formula, is to first do Hamiltonian perturbations to $L_0$ and $L_1$ (as objects in $Fuk(M)$) as in Section \ref{FLTH} (a), and then replace the resulting Lagrangians by their sufficiently small tame perturbations. In Section\ref{df, L_V}, we choose certain conical perturbations to $L_{x,F}$ and $L_{V}$, which combines the Hamiltonian and tame perturbations together.
 
The relative $Pin$-structures on the Lagrangian branes enable us to define orientations on the moduli space of discs, and gives the (higher) compositions over $\mathbb{C}$.
\begin{align*}
\mu^d: &Hom_{Fuk(M)}(L_{d-1}, L_{d})\otimes\cdots Hom_{Fuk(M)}(L_1, L_2)\otimes Hom_{Fuk(M)}(L_0, L_1)\\
&\rightarrow Hom_{Fuk(M)}(L_0, L_{d})[2-d]\\
&\mu^d(\phi_{d-1}\otimes a_{d-1},\cdots, \phi_1\otimes a_1, \phi_0\otimes a_0)=\sum\limits_{a_d\in L_0\cap L_d} \sum\limits_{u\in\mathcal{M}}\text{sgn}(u)\cdot  \phi_u\otimes a_d,
\end{align*}
where $\mathcal{M}= \mathcal{M}(a_0, a_1,\cdots, a_d; L_0, L_1, \cdots, L_d)^{0\text{-d}}$, $\phi_i\in Hom(\mathcal{E}_i|_{a_i}, \mathcal{E}_{i+1}|_{a_i})$ for $i=0,...,d-1$, and $\phi_u\in Hom(\mathcal{E}_0|_{a_d}, \mathcal{E}_d|_{a_d})$ associated to a holomorphic disc $u$ is the composition of successive parallel transport along the edges of $u$ and $\phi_i$ on the corresponding vertices.

\end{document}